\pdfoutput=1
\documentclass[a4paper]{amsart}

\RequirePackage{amsmath}
\RequirePackage{bm}
\RequirePackage{amssymb}
\RequirePackage{upref}
\RequirePackage{amsthm}
\RequirePackage{enumerate}
\RequirePackage{pb-diagram}
\RequirePackage{amsfonts}
\RequirePackage[mathscr]{eucal}
\RequirePackage{verbatim}
\RequirePackage{xr}

\def\@thm#1#2#3{%
  \ifhmode\unskip\unskip\par\fi
  \normalfont
  \trivlist
  \let\thmheadnl\relax
  \let\thm@swap\@gobble
  \let\thm@indent\indent 
  \thm@headfont{\scshape}
  \thm@notefont{}%
  \thm@headpunct{.}
  \thm@headsep 5\p@ plus\p@ minus\p@\relax
  \thm@preskip\topsep
  \thm@postskip\thm@preskip
  #1
  \@topsep \thm@preskip               
  \@topsepadd \thm@postskip           
  \def\@tempa{#2}\ifx\@empty\@tempa
    \def\@tempa{\@oparg{\@begintheorem{#3}{}}[]}%
  \else
    \refstepcounter{#2}%
    \def\@tempa{\@oparg{\@begintheorem{#3}{\csname the#2\endcsname}}[]}%
  \fi
  \@tempa
}

\newcommand{\e}{\epsilon}

\newcommand{\p}{\pi}

\newcommand{\s}{\sigma}

\newcommand{\const}{\tup{const}}

\newcommand{\msp[1]}[1]{\mspace{#1mu}}

\newcommand{\R}[1][n+1]{{\protect\mathbb R}^{#1}}

\newcommand{\N}{{\protect\mathbb N}}

\newcommand{\eR}{\stackrel{\lower1ex \hbox{\rule{6.5pt}{0.5pt}}}{\msp[3]\R[]}}
\newcommand{\eN}{\stackrel{\lower1ex \hbox{\rule{6.5pt}{0.5pt}}}{\msp[1]\N}}
\newcommand{\eO}{\stackrel{\lower1ex
\hbox{\rule{6pt}{0.5pt}}}{\msc O}}

\DeclareMathOperator{\graph}{graph}

\DeclareMathOperator{\osc}{osc}

\newcommand{\uu}{\cup}

\newcommand{\uuu}{\bigcup}

\newcommand{\uud}{ \stackrel{\lower 1ex \hbox {.}}{\uu}}
\newcommand{\uuud}[1]{ \stackrel{\lower 1ex \hbox {.}}{\uuu_{#1}}}

\newcommand{\sminus}[1][28]{\raise 0.#1ex\hbox{$\scriptstyle\setminus$}}

\newcommand{\abs}[1]{\lvert#1\rvert}

\newcommand{\norm}[1]{\lVert#1\rVert}

\newcommand{\tup}{\textup}

\newcommand{\msc}{\protect\mathscr}

\providecommand{\bysame}{\makeboc[3em]{\hrulefill}\thinspace}

\newcommand{\bt}{\begin{thm}}
\newcommand{\bl}{\begin{lem}}
\newcommand{\bc}{\begin{cor}}
\newcommand{\bd}{\begin{definition}}
\newcommand{\bpp}{\begin{prop}}
\newcommand{\br}{\begin{rem}}
\newcommand{\bn}{\begin{note}}
\newcommand{\be}{\begin{ex}}
\newcommand{\bes}{\begin{exs}}
\newcommand{\bb}{\begin{example}}
\newcommand{\bbs}{\begin{examples}}
\newcommand{\ba}{\begin{axiom}}

\newcommand{\et}{\end{thm}}
\newcommand{\el}{\end{lem}}
\newcommand{\ec}{\end{cor}}
\newcommand{\ed}{\end{definition}}
\newcommand{\epp}{\end{prop}}
\newcommand{\er}{\end{rem}}
\newcommand{\en}{\end{note}}
\newcommand{\ee}{\end{ex}}
\newcommand{\ees}{\end{exs}}
\newcommand{\eb}{\end{example}}
\newcommand{\ebs}{\end{examples}}
\newcommand{\ea}{\end{axiom}}

\newcommand{\bp}{\begin{proof}}
\newcommand{\ep}{\end{proof}}
\newcommand{\eps}{\renewcommand{\qed}{}\end{proof}}

\newcommand{\bal}{\begin{align}}

\newcommand{\bi}[1][1.]{\begin{enumerate}[\upshape #1]}
\newcommand{\bia}[1][(1)]{\begin{enumerate}[\upshape #1]}
\newcommand{\bin}[1][1]{\begin{enumerate}[\upshape\bfseries #1]}
\newcommand{\bir}[1][(i)]{\begin{enumerate}[\upshape #1]}
\newcommand{\bic}[1][(i)]{\begin{enumerate}[\upshape\hspace{2\cma}#1]}
\newcommand{\bis}[2][1.]{\begin{enumerate}[\upshape\hspace{#2\parindent}#1]}
\newcommand{\ei}{\end{enumerate}}

\newcommand\ndots{\raise 0.47ex \hbox {,}\hskip0.06em\cdots %
     \raise 0.47ex \hbox {,}\hskip0.06em} 

\newcommand{\hp}{\hphantom}

\newcommand\nd{\noindent}

\newskip\Csmallskipamount                                                
\Csmallskipamount=\smallskipamount
\newskip\Cmedskipamount
\Cmedskipamount=\medskipamount
\newskip\Cbigskipamount
\Cbigskipamount=\bigskipamount

\newcommand\cvs{\vspace\Csmallskipamount}   
\newcommand\cvm{\vspace\Cmedskipamount}

\newskip\csa
\csa=\smallskipamount

\newskip\cma
\cma=\medskipamount

\newskip\cba
\cba=\bigskipamount

\newdimen\spt
\newcommand\citem{\cvs\advance\itemno by
1{(\romannumeral\the\itemno})\hskip3pt}
\newcommand{\bitem}{\cvm\nd\advance\itemno by
1{\bf\the\itemno}\hspace{\cma}}

\newcount\itemno
\usepackage[german,english]{babel}
\usepackage{graphicx}
\RequirePackage{amsmath}
\RequirePackage{bm}
\RequirePackage{amssymb}
\RequirePackage{upref}
\RequirePackage{amsthm}
\RequirePackage{enumerate}
\RequirePackage{pb-diagram}
\RequirePackage{amsfonts}
\RequirePackage[mathscr]{eucal}
\RequirePackage{verbatim}
\RequirePackage{xr}
\RequirePackage{graphicx}
\usepackage{calc}
\usepackage{xspace}

\makeatletter
\RequirePackage{color}
\newcommand{\ann}[1]{\renewcommand{\@makefnmark}{\mbox{$^{\color{red}{\@thefnmark}}$}}%
\footnote {#1}}
\makeatother

\RequirePackage{upref}
\RequirePackage{amsthm}
\RequirePackage{enumerate}
\usepackage[mathscr]{eucal}

\usepackage{xr-hyper}

\usepackage{calc}

\newlength{\oddsidemarginlength}
\newlength{\topmarginlength}

\newcounter{numberoflines}
\newcounter{tempcc}
\oddsidemargin=\oddsidemarginlength
\evensidemargin=\oddsidemargin
\topmargin=\topmarginlength
\usepackage[colorlinks=true,linkcolor=blue,citecolor=blue,urlcolor=blue]{hyperref}  
\usepackage{amsthm}

\begin{document}
\clearpage\pagenumbering{arabic}


\title[Non-scale-invariant ICF in hyperbolic space]{Non-scale-invariant inverse curvature flows in hyperbolic space}

\author{Julian Scheuer}


%
\subjclass[2000]{35J60, 53C21, 53C44, 53C50, 58J05.}
\keywords{curvature flows, inverse curvature flows, hyperbolic space.}
\thanks{This work has been supported by the DFG}
\date{\today}
\address{Ruprecht-Karls-Universit\"at, Institut f\"ur Angewandte Mathematik,
Im Neuenheimer Feld 294, 69120 Heidelberg, Germany}
\email{scheuer@math.uni-heidelberg.de}

%


\swapnumbers

\newtheorem{props}{Proposition}[section]
\newtheorem{cors}[props]{Corollary}
\newtheorem{rems}[props]{Remark}
\newtheorem{lems}[props]{Lemma}
\newtheorem{thms}[props]{Theorem}

\theoremstyle{definition}
\newtheorem{defns}[props]{Definition}

\newtheorem{propss}{Proposition}[
section]
\newtheorem{corss}[propss]{Corollary}
\newtheorem{remss}[propss]{Remark}
\newtheorem{lemss}[propss]{Lemma}
\newtheorem{thmss}[propss]{Theorem}
\newtheorem{definition}[propss]{Definition}
\newtheorem{thm}[propss]{Theorem}

\numberwithin{equation}{section}

\renewcommand{\H}{\mathbb{H}}
\renewcommand{\R}{\mathbb{R}}
\renewcommand{\S}{\mathbb{S}}
\renewcommand{\k}{\kappa}
\newcommand{\del}{\partial}
\renewcommand{\e}{\epsilon}
\renewcommand{\~}{\tilde}
\renewcommand{\-}{\bar}
\renewcommand{\a}{\alpha}
\renewcommand{\b}{\beta}
\newcommand{\g}{\gamma}
\renewcommand{\d}{\delta}
\renewcommand{\t}{\vartheta}
\renewcommand{\p}{\varphi}

\renewcommand{\O}{\mathcal{O}}

\newcommand{\eq}{\begin{equation}}
\newcommand{\eeq}{\end{equation}}

\hspace{10 cm}

\parindent 0 pt

\begin{abstract} 
We consider inverse curvature flows in hyperbolic space $\H^{n+1}$ with starshaped initial hypersurface, driven by positive powers of a homogeneous curvature function.\ \\ The solutions exist for all time and, after rescaling, converge to a sphere.

\end{abstract}

\maketitle

\makeatletter
\def\l@subsection{\@tocline{2}{0pt}{2.5pc}{5pc}{}} 
\makeatother

\tableofcontents

\vspace{0,0001 cm}

\section{Introduction}
During the last decades geometric flows have been studied intensively. Following the ground breaking work of Huisken, \cite{gh:mean}, who considered the mean curvature flow, several authors started to investigate inverse, or expanding flows, e.g. \cite{cg90}, in which nonconvex hypersurfaces were shown to be driven into spheres.
This work, as well as \cite{cg:ImcfHyp}, heavily relied on the homogeneity of the curvature function, leading to, at least in Euclidean space, scale invariance of the flow. In both of these settings, the spherical flows exist for all time and thus dictate the behavior of the solution.\ \\
In  \cite{cg:Impf} an inverse flow driven by arbitrary positive powers of a homogeneous curvature function was considered in $\R^{n+1}$ and for $p>1$ blow up in finite time was proven.\ \\
In the present work we also consider this kind of flow,
$$\dot{x}=F^{-p}\nu,\ 0<p<\infty,$$
in hyperbolic space $\H^{n+1},$ $n\geq 2.$ For $p=1$ this has been treated in \cite{cg:ImcfHyp}, as well as in \cite{ding:hyperbolic} for mean curvature, however in the latter work the obtained convergence results are of less strength.  This flow behaves quite differently compared to the Euclidian case, since the curvature of a geodesic sphere is bounded below by $1,$ so that the flow exists for all time, regardless of the value of $p.$\ \\
In order to formulate the main result of this work, we first need a definition.

\begin{defns} \label{admissable}
Let $\Gamma\subset\R^{n}$ be an open, symmetric and convex cone and $F\in C^{\infty}(\Gamma)$ a symmetric function. A hypersurface $M_{0}\subset\H^{n+1}$ is called \textit{F-admissable}, if 
at any point $x\in M_{0}$ the principal curvatures of $M_{0},$ $\k_{1},...,\k_{n},$ are contained in $\Gamma.$
\end{defns}

We now state our main result.

\begin{thms} \label{mainthm}
Let $\Gamma\subset\R^{n}$ be a symmetric, convex and open cone, such\hspace{0,1 cm} that 
\eq \Gamma_{+}=\{(\k_{i})\in\R^{n}\colon \k_{i}>0\ \forall 1\leq i\leq n\}\subset\Gamma \eeq
and $F\in C^{\infty}(\Gamma)\cap C^{0}(\-{\Gamma})$ be a monotone, $1$-homogeneous and concave curvature function, such that 
\eq F_{|\Gamma}>0, F_{|\del\Gamma}=0\ \ \and\ F(1,...,1)=n. \eeq 
Let $p>0$ and in case $p>1$ suppose $\Gamma=\Gamma_{+}.$ Let $M\hookrightarrow M_{0}\subset \H^{n+1}$ be a smooth and $F$-admissable embedded closed hypersurface, which can be written as a graph over a geodesic sphere, identified with $\S^{n},$
\eq M_{0}=\graph u(0,\cdot).\eeq Then\ \\
(1) there is a unique smooth curvature flow $$x\colon [0,\infty)\times M\rightarrow \H^{n+1},$$ which satisfies the flow equation
\begin{align} \begin{split} \label{floweq} \dot{x}&=-\Phi(F)\nu, \\
						x(0)&=M_{0}, \end{split} \end{align}
where $\nu(t,\xi)$ is the outward normal to $M_{t}=x(t,M)$ at $x(t,\xi),$ $F$ is evaluated at the principal curvatures of $M_{t}$ in $x(t,\xi),$  
\eq \Phi(r)=-r^{-p} \eeq and the leaves $M_{t}$ are graphs over $\S^{n},$ 
\eq M_{t}=\graph u(t,\cdot). \eeq
(2) 
 For all $0<p\leq 1$ the leaves $M_{t}$ become more and more umbilic, namely
\eq |h^{i}_{j}-\delta^{i}_{j}|\leq ce^{-\frac{2}{n^{p}}t},\ c=c(n, p,M_{0}). \eeq
In case $p>1$ there exists $\e=\e(n, p, M_{0}),$ such that the same conclusion holds, if we impose the $C^{0}$-pinching condition 
\eq \osc u(0,\cdot)<\e. \eeq
(3) Under the appropriate conditions as in (2) we obtain, that the rescaled surfaces 
\eq \hat{M}_{t}=\graph\left(u-\frac{t}{n^{p}}\right) \eeq
converge to a well defined, smooth function in $C^{\infty}$ and thus the rescaled surfaces
\eq \~{M}_{t}=\graph\frac{u}{t}\eeq
converge to a geodesic sphere in $C^{\infty}.$

\end{thms}

\section{Setting and general facts}
We now state some general facts about hypersurfaces, especially those that can be written as graphs. We basically follow the description of \cite{cg:ImcfHyp}, but restrict to Riemannian manifolds. For a detailed discussion we refer to \cite{cg:cp}. \ \\
Let $N=N^{n+1}$ be Riemannian and $M=M^{n}\hookrightarrow N$ be a hypersurface. The geometric quantities of $N$ will be denoted by $(\-{g}_{\a\b}),$ $(\-{R}_{\a\b\g\d})$ etc., where greek indices range from $0$ to $n$. Coordinate systems in $N$ will be denoted by $(x^{\a}).$\ \\
Quantities for $M$ will be denoted by $(g_{ij}),$ $(h_{ij})$ etc., where latin indices range from $1$ to $n$ and coordinate systems will generally be denoted by $(\xi^{i}),$ unless stated otherwise.\ \\
Covariant differentiation will usually be denoted by indices, e.g. $u_{ij}$ for a function $u\colon M\rightarrow \R$, or, if ambiguities are possible, by a semicolon, e.g.
$h_{ij;k}.$ Usual partial derivatives will be denoted by a comma, e.g. $u_{i,j}.$\ \\
Let $x\colon M\hookrightarrow N$ be an embedding and $(h_{ij})$ be the second fundamental form, then we have the \textit{Gaussian formula}
\eq x^{\a}_{ij}=-h_{ij}\nu^{\a}, \eeq where $\nu$ is a differentiable normal, the \textit{Weingarten equation}
\eq \nu^{\a}_{i}=h^{k}_{i}x^{\a}_{k}, \eeq
the \textit{Codazzi equation} \eq h_{ij;k}-h_{ik;j}=\-{R}_{\a\b\g\d}\nu^{\a}x^{\b}_{i}x^{\g}_{j}x^{\d}_{k} \eeq
and the \textit{Gau\ss\ equation} \eq R_{ijkl}=(h_{ik}h_{jl}-h_{il}h_{jk})+\-{R}_{\a\b\g\d}x^{\a}_{i}x^{\b}_{j}x^{\g}_{k}x^{\d}_{l}. \eeq
Since in our case $N=\H^{n+1},$ we have 
\eq \-{R}_{\a\b\g\d}=\-{g}_{\a\d}\-{g}_{\b\g}-\-{g}_{\a\g}\-{g}_{\b\d} \eeq
and thus the Codazzi equation takes the form \eq h_{ij;k}=h_{ik;j}. \eeq
Now assume that $N=(a,b)\times S_{0},$ where $S_{0}$ is compact Riemannian and that there is a Gaussian coordinate system $(x^{\a})$ such that
\eq d\-{s}^{2}=e^{2\psi}((dx^{0})^{2}+\sigma_{ij}(x^{0},x)dx^{i}dx^{j}), \eeq
where $\sigma_{ij}$ is a Riemannian metric, $x=(x^{i})$ are local coordinates for $S_{0}$ and $\psi\colon N\rightarrow \R$ is a function.\ \\
Let $M=\graph u_{|S_{0}}$ be a hypersurface
\eq M=\{(x^{0},x)\colon x^{0}=u(x), x\in S_{0}\}, \eeq
then the induced metric has the form \eq g_{ij}=e^{2\psi}(u_{i}u_{j}+\sigma_{ij}) \eeq
with inverse \eq g^{ij}=e^{-2\psi}(\sigma^{ij}-v^{-2}u^{i}u^{j}), \eeq
where $(\sigma^{ij})=(\sigma_{ij})^{-1},$ $u^{i}=\sigma^{ij}u_{j}$ and \eq v^{2}=1+\sigma^{ij}u_{i}u_{j}\equiv 1+|Du|^{2}.\eeq
We use, especially in the Gaussian formula, the normal 
\eq \label{outernormal} (\nu^{\a})=v^{-1}e^{-\psi}(1,-u^{i}). \eeq
Looking at $\a=0$ in the Gaussian formula, we obtain 
\eq e^{-\psi}v^{-1}h_{ij}=-u_{ij}-\-{\Gamma}^{0}_{00}u_{i}u_{j}-\-{\Gamma}^{0}_{0i}u_{j}-\-{\Gamma}^{0}_{0j}u_{i}-\-{\Gamma}^{0}_{ij} \eeq
and \eq e^{-\psi}\-{h}_{ij}=-\-{\Gamma}^{0}_{ij}, \eeq
where covariant derivatives are taken with respect to $g_{ij}.$\ \\
Let us state some properties of $\H^{n+1}.$ $\H^{n+1}$ is parametrizable over $B_{2}(0)$ yielding the conformally flat metric
\eq d\-{s}^{2}=\frac{1}{(1-\frac 14 r^{2})^{2}}(dr^{2}+r^{2}\sigma_{ij}dx^{i}dx^{j}),\eeq 
where $(\sigma_{ij})$ is the canonical metric of $\S^{n},$ cf. \cite[p.16]{cg:ImcfHyp}. Also compare \cite[Thm. 10.2.1]{cg:cp}.\ \\
Defining $\tau$ by \eq \tau= \log(2+r)-\log(2-r),  \eeq
such that \eq d\tau=\frac{1}{1-\frac 14 r^{2}}dr, \eeq
then \eq d\-{s}^{2}=d\tau^{2}+\sinh^{2}\tau\sigma_{ij}dx^{i}dx^{j}. \eeq
Thus we have a parametrization of $\H^{n+1}$ over $\R^{n+1}$ and, using \cite[Thm. 1.7.5]{cg:cp}, we see that in geodesic polar coordinates around a given point the metric takes the above representation. In the sequel we will again write $r$ for $\tau,$ for greater clarity.\ \\
The geodesic spheres are totally umbilic and, setting \eq \-{g}_{ij}=\sinh^{2}r\sigma_{ij}, \eeq
their second fundamental form is given by \eq \-{h}_{ij}=\coth r \-{g}_{ij}. \eeq
Thus $\-{h}^{i}_{j}=\coth r\d^{i}_{j}$ and $\-{\k}_{i}=\coth r.$
The second fundamental form of a graph $M=\graph u$ satisfies
\eq \label{graphA} h_{ij}v^{-1}=-u_{ij}+\-{h}_{ij}. \eeq

\section{Long time existence}
\subsection*{$C^{0}$-estimates}\ \\

We first construct the spherical barriers of the flow.

\begin{propss}
Consider (\ref{floweq}) with $x(0)=S_{r_{0}}=\{x^{0}=r_{0}\}.$ Then the corresponding flow $x=x(t,\xi)$ exists for all time. The leaves $M(t)=x(t,M)$ are geodesic spheres with radius 
\eq x^{0}(t,M)=\Theta(t,r_{0}), \eeq
where $\Theta$ solves the ODE
\begin{align} \begin{split} \label{sphFlow} \dot{\Theta} = F^{-p}&=n^{-p}\coth^{-p}\Theta \\
										\Theta(0,r_{0})&=r_{0}. \end{split} \end{align}
\end{propss}

\begin{proof}
Looking at (\ref{outernormal}), we see that the outer normal of a geodesic sphere is $(1,0,...,0)$ and thus, setting
\begin{align} \begin{split} x^{0}(t,\xi)&=\Theta(t,r_{0}) \\
					x^{i}(t,\xi)&=x^{i}(0,\xi), \end{split} \end{align}
where $\Theta$ is the unique solution of (\ref{sphFlow}), we see that $x$ solves the flow equation, also using that $F(\-{h}^{i}_{j})=n\coth\Theta.$			The solution of the ODE exists for all time, since $0<\dot{\Theta}\leq n^{-p}.$		
\end{proof}

We now derive further properties of the spherical flows.

\begin{propss} \label{sphFlowBound}
Let $\Theta_{i}=\Theta(t,r_{i}),$ $i=1,2,$ be solutions of (\ref{sphFlow}), $r_{1}<r_{2},$ then\ \\
\eq r_{i}+\frac{t}{n^{p}\coth^{p}r_{i}}\leq \Theta_{i}(t)\leq r_{i}+\frac{t}{n^{p}} \eeq and there exists $c=c(r_{1},n,p),$ such that
\eq \label{sphFlowBound1}  0< \Theta_{2}(t)-\Theta_{1}(t)\leq c(r_{2}-r_{1})\ \ \forall t\in [0,\infty) \eeq
and such that $p\mapsto c(r_{1},n,p)$ is continuous.
\end{propss}

\begin{proof}
The first inequality follows from \eq \frac{1}{n^{p}\coth^{p}r_{i}}\leq \dot{\Theta}\leq \frac{1}{n^{p}}, \eeq
since $\dot{\Theta}>0$ and since $\coth$ is decreasing.
To prove the second claim, define 
\eq \rho(t)=\Theta_{2}(t)-\Theta_{1}(t). \eeq
$\rho$ is positive, since this is the case at $t=0$ and different orbits of an ODE flow can not intersect.
We have \begin{align} \begin{split} \dot{\rho}(t)&=\frac{1}{n^{p}\coth^{p}\Theta_{2}}-\frac{1}{n^{p}\coth^{p}\Theta_{1}} \\
									&\leq \frac{1}{n^{p}}(\coth^{p}\Theta_{1}-\coth^{p}\Theta_{2}) \\
									&=\frac{1}{n^{p}}(p \coth^{p-1}(s)(\coth^{2}(s)-1))(\Theta_{2}-\Theta_{1}),\ s\in[\Theta_{1}(t),\Theta_{2}(t)]\\
									&\leq \~{c}(n,p,r_{1})(\coth^{2}\Theta_{1}-1)\rho(t)\\
									&=\~{c}\sinh^{-2}\Theta_{1}\rho(t)\\
									&\leq \~{c}\sinh^{-2}(r_{1}+ct)\rho(t), \end{split}\end{align}	
Thus \begin{align}\begin{split} \log \rho(t)&\leq \log \rho(0)+\int_{0}^{t}\~{c}\sinh^{-2}(cs+r_{1})ds \\
										&=\log(r_{2}-r_{1})+\frac{\~{c}}{c}[-\coth(cs+r_{1})]^{t}_{0}\\
										&\leq \log(r_{2}-r_{1})+\frac{\~{c}}{c}\coth r_{1} \end{split}\end{align} and
									$$ \rho(t)\leq c(n,p,r_{1})(r_{2}-r_{1}).$$
\end{proof}

\begin{corss} \label{sphFlowGrowth}
Let $\Theta=\Theta(t,r_{0})$ be a solution of (\ref{sphFlow}), then there exists $c=c(r_{0},n,p),$ such that
\eq -c<\Theta -\frac{t}{n^{p}}<c\ \ \forall t\in[0,\infty). \eeq
\end{corss}

\begin{proof}
The upper estimate follows from Proposition \ref{sphFlowBound} immediately.
There holds
\begin{align}\begin{split} \dot{\Theta}-\frac{1}{n^{p}}&=\frac{1}{n^{p}\coth^{p}\Theta}-\frac{1}{n^{p}}\\
										&=\frac{1}{n^{p}}\frac{1-\coth^{p}\Theta}{\coth^{p}\Theta}\\
										&\geq\frac{1}{n^{p}}(1-\coth^{p}\Theta)\\
										&\geq \frac{1}{n^{p}}(1-\coth^{m}\Theta), \ p\leq m\in\mathbb{Z}\\
										&=\frac{1}{n^{p}}\sum_{k=0}^{m-1}\coth^{k}\Theta(1-\coth\Theta)\\
										&\geq c(n,p,r_{0})(1-\coth^{2}\Theta)\\
										&\geq c(1-\coth^{2}(r_{0}+\~{c}t)) \end{split}\end{align}
	and thus
\begin{align}\begin{split}\Theta(t)-\frac{t}{n^{p}}&\geq r_{0}+c\int_{0}^{t}(1-\coth^{2}(r_{0}+\~{c}s))ds\\
												&=r_{0}+\frac{c}{\~{c}}[\coth(r_{0}+\~{c}s)]^{t}_{0}\\
												&=r_{0}+\frac{c}{\~{c}}\coth(r_{0}+\~{c}t)-\frac{c}{\~{c}}\coth r_{0}\\
												&\geq r_{0}-\frac{c}{\~{c}}\coth r_{0} \end{split}\end{align} 
\end{proof}

\begin{remss} \label{shorttime}
Looking at \cite[Thm. 2.5.19]{cg:cp} and \cite[Thm. 2.6.1]{cg:cp}, under the assumptions of Theorem \ref{mainthm} we obtain short time existence of the flow on a maximal interval $[0,T^{*}),$ $0<T^{*}\leq\infty,$ and 
\eq x\in C^{\infty}([0,T^{*})\times M,\H^{n+1}). \eeq
This includes, that all the leaves $M(t)=x(t,M),$ $0\leq t<T^{*},$ are admissable in the sense of Definition \ref{admissable} and can be written as graphs over $\S^{n}.$ Furthermore the flow $x$ exists as long as the scalar flow
\eq \label{scalarFlow} \dot{u}=\frac{\del u}{\del t}=-\Phi v \eeq
does, where \eq u\colon [0,T^{*})\times \S^{n}\rightarrow \R, \eeq also compare \cite[Thm. 2.5.17]{cg:cp} and \cite[p. 98-99]{cg:cp}. Thus, for the rest of the next chapters we will most of the time investigate long time existence for (\ref{scalarFlow}).
\end{remss}

\vspace{0,05 cm}

\begin{lemss} \label{oscBound}
The solution $u$ of (\ref{scalarFlow}) satisfies
\eq \label{oscBound1}\Theta(t,\inf u(0,\cdot))\leq u(t,x)\leq \Theta(t,\sup u(0,\cdot))\ \ \forall t\in[0,T^{*}) \ \forall x\in \S^{n}. \eeq

In particular we have 
\eq \osc u(t,\cdot)=\sup u(t,\cdot)-\inf u(t,\cdot)\leq c\osc u(0,\cdot), \eeq
$c=c(n,p,\inf u(0,\cdot)).$
\end{lemss}

\begin{proof}
Let \eq w(t)=\sup u(t,\cdot)=u(t,x_{t}). \eeq
By \cite[Lemma 6.3.2]{cg:cp}, $w$ is Lipschitz continuous and at a point of differentiability we have
\begin{align} \begin{split} \dot{w}(t)=\frac{\del u}{\del t}(t,x_{t})&=\frac{1}{F^{p}(-g^{ik}u_{kj}+\-{h}^{i}_{j})}\\
												&\leq \frac{1}{F^{p}(\-{h}^{i}_{j})}\\
												&=\frac{1}{n^{p}\coth^{p}w}\equiv \mathcal{L}(w) \end{split} \end{align}
On the other hand
\eq \dot{\Theta}(\cdot,\sup u(0,\cdot))=\mathcal{L}(\Theta(\cdot,\sup u(0,\cdot))) \eeq
as well as \eq w(0)=\Theta(0,\sup u(0,\cdot)), \eeq
from which the upper estimate follows by integration and Gronwall's lemma applied to $w-\Theta.$
The estimate from below follows identically.
\end{proof}

\begin{corss} \label{thetaGrowth}
Define \eq \t(r)=\sinh r \eeq and let $u$ be the solution of (\ref{scalarFlow}). Then there exists $c=c(n,p, M_{0}),$ such that
\eq \label{thetaGrowth1} 0<c^{-1}\leq \t(u)e^{-\frac{t}{n^{p}}}\leq c\ \ \forall t\in [0,T^{*}). \eeq
and \eq \frac{\-{H}(u)}{n}-1=\coth u-1\leq ce^{-\frac{2}{n^{p}}t}. \eeq
\end{corss}

\begin{proof}
We deduce
\begin{align} \begin{split} \t(u)e^{-\frac{t}{n^{p}}}&=\frac 12 \left(e^{u-\frac{t}{n^{p}}}-e^{-(u+\frac{t}{n^{p}})}\right)\\
									&\leq \frac 12 e^{\Theta(t,\sup u(0,\cdot))-\frac{t}{n^{p}}}\\
									&\leq c(\sup u(0,\cdot),n,p), \end{split} \end{align}
		by Corollary \ref{sphFlowGrowth}, as well as
\begin{align} \begin{split} \t(u)e^{-\frac{t}{n^{p}}}&\geq \frac 12 \left(e^{\Theta(t,\inf u(0,\cdot))-\frac{t}{n^{p}}}-e^{-\Theta(t,\inf u(0,\cdot))-\frac{t}{n^{p}}}\right)\\
									&=\frac 12 \left(e^{\Theta(t,\inf u(0,\cdot))-\frac{t}{n^{p}}}-e^{\Theta(t,\inf u(0,\cdot))-\frac{t}{n^{p}}-2\Theta (t,\inf u(0,\cdot))}\right) \\
									&\geq \frac 12 e^{-c}\left(1-e^{-2\Theta(t,\inf u(0,\cdot))}\right)\\
									&\geq c>0. \end{split}\end{align}
	Furthermore
	\begin{align}\begin{split} \coth u-1&= \frac{\cosh u-\sinh u}{\t(u)}=\frac{e^{-u}}{\t(u)} \\
								&= \frac{e^{-(u-\frac{t}{n^{p}})}e^{-\frac{2}{n^{p}}t}}{\t(u)e^{-\frac{t}{n^{p}}}}\\
								&\leq \frac{e^{-\Theta(t,\inf u(0,\cdot))+\frac{t}{n^{p}}}e^{-\frac{2}{n^{p}}t}}{\t(u)e^{-\frac{t}{n^{p}}}}\\
								&\leq c e^{-\frac{2}{n^{p}}t}. \end{split}\end{align}		
\end{proof}

\subsection*{$C^{1}$-estimates}
\begin{lemss} \label{vconvex}
Let $u$ be the short time solution of (\ref{scalarFlow}) in case $p>1.$ Then for the quantity
\eq v=\sqrt{1+\-{g}^{ij}u_{i}u_{j}}\equiv \sqrt{1+|Du|^{2}} \eeq
there exists $c=c(n,p,M_{0}),$ such that
\eq v\leq c\ \ \forall t\in [0,T^{*}). \eeq
Furthermore $c$ depends on $p$ continuously.
\end{lemss}

\begin{proof}
In case $p>1$ the leaves $M(t)$ are convex. Thus, \cite[Thm. 2.7.10]{cg:cp}, especially estimate (2.7.83)
\eq \label{convexitygradbound} v\leq e^{\bar{\kappa}(\sup u-\inf u)}\eeq
 is applicable. Note that in this estimate, an upper bound for the principal curvatures of $\{x^{0}=\const\}$ is uniformly given by some $\-{\k}=\-{\k}(\inf u(0,\cdot)).$ Thus we obtain the claim in view of Lemma \ref{oscBound}.
\end{proof}

In case $p\leq 1$ we do not assume convexity. We use the maximum principle to estimate $v$.\\

 We follow the method in \cite{cg:ImcfHyp}.

\begin{remss} \label{phi}
Defining \eq \p=\int_{r_{0}}^{u}\t^{-1} \eeq
and having (\ref{graphA}) in mind, we obtain
\eq \label{phiA} h^{i}_{j}=g^{ik}h_{kj}=v^{-1}\t^{-1}(-(\sigma^{ik}-v^{-2}\p^{i}\p^{k})\p_{jk}+\dot{\t}\d^{i}_{j}), \eeq
where covariant differentiation and index raising happens with respect to $\sigma_{ij},$ cf. \cite[(3.26)]{cg:ImcfHyp}. We obtain
\eq \label{phiScalarFlow} \dot{\p}=\t^{-1}\dot{u}=\frac{\t^{p-1}v}{F^{p}(\t h^{i}_{j})}\equiv \frac{\t^{p-1}v}{F^{p}(\~{h}^{i}_{j})}. \eeq
There holds
\eq g_{ij}=u_{i}u_{j}+\t^{2}\sigma_{ij}=\t^{2}(\p_{i}\p_{j}+\sigma_{ij})\equiv \t^{2}\~{g}_{ij}. \eeq
Defining \eq \~{h}_{ij}=\~{g}_{ik}\~{h}^{k}_{j}, \eeq
we see that in (\ref{phiScalarFlow}) we are considering the eigenvalues of $\~{h}_{ij}$ with respect to $\~{g}_{ij}$ and thus we define 
\eq F^{ij}=\frac{\del F}{\del \~{h}_{ij}}\ \mathrm{and}\ F^{i}_{j}=\frac{\del F}{\del \~{h}_{i}^{j}}. \eeq
We have \eq \~{h}_{ij}=\~{g}_{ik}\~{h}^{k}_{j}=\t^{-2}g_{ik}\t h^{k}_{j}=\t^{-1}h_{ij}, \eeq
hence $\~{h}_{ij}$ is symmetric. Furthermore note
\eq |Du|^{2}=\sigma^{ij}\p_{i}\p_{j}\equiv|D\p|^{2}, \eeq
as well as
\eq \label{phiA2} \~{h}^{l}_{k}=-v^{-1}\~{g}^{lj}\p_{jk}+v^{-1}\dot{\t}\d^{l}_{k}. \eeq
\end{remss}

\begin{lemss} \label{phiDer}
The various quantities and tensors in (\ref{phiScalarFlow}) satisfy
\eq (\t^{p-1})_{i}=(p-1)\t^{p-1}\dot{\t}\p_{i},\eeq
\eq v_{i}=v^{-1}\p_{ki}\p^{k}, \eeq
\eq \~{g}^{lr}_{\hp{lr};i}=2v^{-3}v_{i}\p^{l}\p^{r}-v^{-2}(\p^{l}_{i}\p^{r}+\p^{l}\p^{r}_{i}) \eeq
and \eq \~{h}^{l}_{k;i}=v^{-2}v_{i}(\~{g}^{lr}\p_{rk}-\dot{\t}\d^{l}_{k})-v^{-1}(\~{g}^{lr}_{\hp{lr};i}\p_{rk}+\~{g}^{lr}\p_{rki}-\t^{2}\p_{i}\d^{l}_{k}), \eeq
where $(\~{g}^{rl})=(\~{g}_{rl})^{-1}$ and the covariant derivatives as well as index raising are performed with respect to $\sigma_{ij}.$
\end{lemss}

\begin{proof}
This is a straightforward computation in any of the cases. Just have in mind that $\t=\t(u),$ such that $\t_{i}=\dot{\t}u_{i}=\dot{\t}\t\p_{i}.$
\end{proof}

\begin{lemss} \label{GradBound}
Let $u$ be the solution of (\ref{scalarFlow}) in case $p\leq 1.$ Then 
\eq \label{GradBound2}  v\leq\sup v(0,\cdot). \eeq
\end{lemss}

\begin{proof}
From Remark \ref{phi} we see, that it suffices to bound $|D\p|^{2}.$ Differentiate \eq\dot{\p}=-\Phi v\t^{p-1},\ \Phi=\Phi(F(\~{h}^{k}_{l})), \eeq
with respect to $\p^{i}D_{i}.$
From Lemma \ref{phiDer} we find, setting \eq w=\frac 12 |D\p|^{2}=\frac 12 \p_{k}\p^{k}, \eeq
\eq \dot{w}=\dot{\p}_{i}\p^{i}=-v\t^{p-1}\dot{\Phi}F^{k}_{l}\~{h}^{l}_{k;i}\p^{i}-\Phi\t^{p-1}v_{i}\p^{i}-(p-1)\Phi v\t^{p-1}\dot{\t}|D\p|^{2}. \eeq
Fix $0<T<T^{*}$ and suppose \eq \sup_{[0,T]\times\S^{n}} w=w(t_{0},x_{0})>0,\ t_{0}>0,\eeq
then at this point we have
\begin{align}\begin{split} \label{GradBound1}0&\leq(p-1)F^{-p}v\t^{p-1}\dot{\t}|D\p|^{2}
							-v\t^{p-1}\dot{\Phi}F^{k}_{l}(-v^{-1}\~{g}^{lr}\p_{rki}\p^{i}\\
							&\hphantom{=}+v^{-1}\t^{2}|D\p|^{2}\d^{l}_{k}+v^{-3}\p_{rk}\p^{r}\p^{l}_{i}\p^{i}+v^{-3}\p^{r}_{i}\p^{i}\p^{l}\p_{rk})\\
					&=2(p-1)F^{-p}v\t^{p-1}\dot{\t}w-2\dot{\Phi}\t^{p+1}F^{kl}\~{g}_{kl}w\\
					&\hphantom{=}+\t^{p-1}\dot{\Phi}F^{kr}\p_{rki}\p^{i}\\
					&=(2(p-1)F^{-p}v\t^{p-1}\dot{\t}-2\dot{\Phi}\t^{p+1}F^{kl}\~{g}_{kl})w\\
					&\hphantom{=}+\t^{p-1}\dot{\Phi}F^{kr}(\p_{irk}+\p_{k}\sigma_{ri}-\p_{i}\sigma_{rk})\p^{i}\\
					&=(2(p-1)F^{-p}v\t^{p-1}\dot{\t}-2\dot{\Phi}\t^{p+1}F^{kl}\~{g}_{kl})w\\
					&\hphantom{=}+\t^{p-1}\dot{\Phi}F^{kr}(\p_{k}\p_{r}-|D\p|^{2}\sigma_{kr})\\
					&\hphantom{=}+\t^{p-1}\dot{\Phi}F^{kr}w_{rk}-\t^{p-1}\dot{\Phi}F^{kr}\p_{ir}\p^{i}_{k}\\
					&<0. \end{split} \end{align}
Hence the estimate (\ref{GradBound2}) is valid, since $T$ is arbitrary.
\end{proof}

\subsection*{Curvature estimates and long time existence}

\begin{propss} \label{curvBound}
Let $x$ be a solution of the curvature flow (\ref{floweq}), $0<p<\infty.$ Then the curvature function is bounded from above and below, i.e. there exists $c=c(n,p,M_{0}),$ such that
\eq  0<c^{-1}\leq F(t,\xi)\leq c<\infty\ \ \forall (t,\xi)\in [0,T^{*})\times M.\eeq
\end{propss}

\begin{proof}
The proof proceeds similarly to the one in \cite[Lemma 4.1]{cg:ImcfHyp}.\\

Define \eq \chi=v\eta(u)\equiv \frac{v}{\sinh u} \eeq
and note \eq \dot{\eta}=-\frac{\-{H}}{n}\eta, \eeq
where $\eta=\eta(r)$ and $\-{H}$ is the mean curvature of $S_{r}.$ Then $\chi$ satisfies
\eq \label{chi} \dot{\chi}-\dot{\Phi}F^{ij}\chi_{ij}=-\dot{\Phi}F^{ij}h_{ik}h^{k}_{j}\chi-2\chi^{-1}\dot{\Phi}F^{ij}\chi_{i}\chi_{j}+(\dot{\Phi}F+\Phi)\frac{\-{H}}{n}v\chi,\eeq
cf. \cite[Lemma 5.8]{cg:spaceform}.
$\Phi,$ and also $-\Phi,$ satisfy
\eq \label{Phi} \Phi'-\dot{\Phi}F^{ij}\Phi_{ij}=\dot{\Phi}F^{ij}h_{ik}h^{k}_{j}\Phi+K_{N}\dot{\Phi}F^{ij}g_{ij}\Phi, \eeq
where $'$ denotes the time derivative of the evolution and $\dot{\Phi}=\frac{d}{dr}\Phi(r),$ cf. \cite[Lemma 2.3.4]{cg:cp}. Note that we have $K_{N}=-1.$
The function $u$ satisfies
\eq \label{u} \dot{u}-\dot{\Phi}F^{ij}u_{ij}=(\dot{\Phi}F-\Phi)v^{-1}-\dot{\Phi}F^{ij}\-{h}_{ij}, \eeq
where $\dot{u}$ is a total derivative, cf. \cite[Lemma 3.3.2]{cg:cp}.\ \\

(i) We first prove $F\geq c>0.$ Set 
\eq \~{\chi}=\chi e^{\frac{t}{n^{p}}}. \eeq
Then there exists $c=c(n,p,M_{0}),$ such that
\eq 0<c^{-1}\leq \~{\chi}(t,\xi)\leq c<\infty\ \ \forall (t,\xi)\in [0,T^{*})\times M, \eeq
where we used Corollary \ref{thetaGrowth} and $v\leq c.$
Set \eq w=\log(-\Phi)+\log \~{\chi}, \eeq
fix $0<T<T^{*}$
and suppose \eq \sup_{[0,T]\times M} w=w(t_{0},\xi_{0})>0,\ t_{0}>0. \eeq
Then in $(t_{0},\xi_{0})$ there holds
\eq \frac{\Phi_{i}}{\Phi}=-\frac{\chi_{i}}{\chi} \eeq and
\eq 0\leq \dot{w}-\dot{\Phi}F^{ij}w_{ij}=-\dot{\Phi}F^{ij}g_{ij}+(\dot{\Phi}F+\Phi)\frac{\-{H}}{n}v+\frac{1}{n^{p}}.\eeq 
Thus\begin{align}\begin{split}  0&\leq -pF^{ij}g_{ij}+(p-1)F\frac{\-{H}}{n}v+n^{-p}F^{p+1}\\
								&\leq -pn+(p-1)F\frac{\-{H}}{n}v+n^{-p}F^{p+1}. \end{split}\end{align}
Moreover \eq \-{H}=n\coth u\leq n\coth\inf u(0,\cdot),\eeq so that
 \eq 0\leq\begin{cases} -pn+n^{-p}F^{p+1}, &0<p\leq 1\\
							-pn+(p-1)\coth\inf u(0,\cdot)Fv+n^{-p}F^{p+1}, &p>1. \end{cases} \eeq
Without loss of generality suppose $w(t_{0},\xi_{0})$ is so large, that $F(t_{0},\xi_{0})<1.$ Then
\eq F(t_{0},\xi_{0})\geq \begin{cases} p^{\frac{1}{p+1}}n, & 0<p\leq 1 \\
							\frac{pn-n^{-p}}{(p-1)cv}, &p>1, \end{cases} \eeq
$c=c(M_{0}).$
Hence, at a point, where $w$ attains a maximum, we have $F\geq c=c(n,p,M_{0}).$ Thus
\eq w\leq w(t_{0},\xi_{0})\leq \log\left(\frac{1}{c^{p}}\right)+c\equiv c(n,p,M_{0})\eeq and
\eq \frac{1}{F^{p}}= e^{w}\~{\chi}^{-1}\leq c(n,p,M_{0}). \eeq
Thus, $F$ is uniformly bounded below in $[0,T^{*}).$ \ \\ 

(ii) We prove $F\leq c.$\ \\
Define \eq \~{u}=u-\frac{t}{n^{p}}. \eeq
Then, by Corollary \ref{sphFlowGrowth} and Lemma \ref{oscBound} we have 
\eq \~{u}>c. \eeq
Set \eq w=-\log(-\Phi)+\~{u}. \eeq
Then, in a maximal point $(t_{0},\xi_{0})\in (0,T]\times M,$ $0<T<T^{*},$ of $w$ we have
\begin{align}\begin{split} 0&\leq \dot{w}-\dot{\Phi}F^{ij}w_{ij}\\ &=-\dot{\Phi}F^{ij}h_{ik}h^{k}_{j}+\dot{\Phi}F^{ij}g_{ij}-\dot{\Phi}F^{ij}(\log(-\Phi))_{i}(\log(-\Phi))_{j}\\	
                         										&\hphantom{=}+(\dot{\Phi}F-\Phi)v^{-1}-\dot{\Phi}F^{ij}\-{h}_{ij}-\frac{1}{n^{p}}\\
										 	&=-\dot{\Phi}F^{ij}h_{ik}h^{k}_{j}+\dot{\Phi}F^{ij}(u_{i}u_{j}+\-{g}_{ij}-\coth u\-{g}_{ij})\\
											&\hphantom{=}-\dot{\Phi}F^{ij}u_{i}u_{j}+(\dot{\Phi}F-\Phi)v^{-1}-\frac{1}{n^{p}}\\
											&\leq (p+1)F^{-p}v^{-1}-\frac{1}{n^{p}}, \end{split}\end{align}
		where we used $\coth u\geq 1$ and $0=w_{i}$ in $(t_{0},\xi_{0}).$ Then
\eq F(t_{0},\xi_{0})\leq c(n,p,M_{0}), \eeq leading to \eq w\leq c(n,p,M_{0}) \eeq and finally
\eq F^{p}\leq e^{w}e^{-\~{u}}\leq c(n,p,M_{0}). \eeq
\end{proof}

\begin{propss} \label{kappaBound}
The leaves $M(t)$ of (\ref{floweq}) have uniformly bounded principal curvatures, i.e. there exists $c=c(n,p,M_{0}),$ such that
\eq \k_{i}(t,\xi)\leq c\ \ \forall (t,\xi)\in [0,T^{*})\times M. \eeq
Thus the principal curvatures stay in a compact set $K=K(n,p,M_{0})\subset\Gamma,$ in view of Proposition \ref{curvBound}. 
\end{propss}

\begin{proof}
Basically, the proof of the corresponding lemma in \cite[Lemma 4.4]{cg:ImcfHyp}, applies in our case with slight modifications.\\

 Since $\H^{n+1}$ has constant curvature $K_{N}=-1,$ we have 
\begin{align}\begin{split} \dot{h}^{i}_{j}-\dot{\Phi}F^{kl}h^{i}_{j;kl}&=\dot{\Phi}F^{kl}h_{rk}h^{r}_{l}h^{i}_{j}+(\Phi-\dot{\Phi}F)h^{ki}h_{kj}+\ddot{\Phi}F_{j}F^{i}\\
												&\hphantom{=}+\dot{\Phi}F^{kl,rs}h_{kl;j}{h_{rs;}}^{i}-(\Phi+\dot{\Phi}F)\d^{i}_{j}+\dot{\Phi}F^{kl}g_{kl}h^{i}_{j}.
												\end{split}\end{align}
Let $\~{\chi}=\chi e^{\frac{t}{n^{p}}}.$ Setting
\eq \hat{\chi}=\~{\chi}^{-1}, \eeq
we find a constant $\theta >0,$ such that
\eq 2\theta \leq \hat{\chi}(t,\xi)\ \ \forall (t,\xi)\in[0,T^{*})\times M. \eeq
Define the functions
\eq \zeta=\sup\{h_{ij}\eta^{i}\eta^{j}\colon \norm{\eta}^{2}=g_{ij}\eta^{i}\eta^{j}=1\}, \eeq
\eq \phi=-\log(\hat{\chi}-\theta) \eeq
and \eq w=\log \zeta +\phi+\lambda \~{u}, \eeq
where $\~{u}=u-\frac{t}{n^{p}},$ and $\lambda$ is to be chosen later.
We wish to bound $w$ from above. Thus, suppose $w$ attains a maximal value at $(t_{0},\xi_{0})\in (0,T]\times M,$ $T<T^{*}.$
Choose Riemannian normal coordinates in $(t_{0},\xi_{0}),$ such that in this point we have
\eq g_{ij}=\d_{ij}\ \wedge\ h_{ij}=\k_{i}\d_{ij}\ \wedge\ \k_{1}\leq \ldots\leq \k_{n}. \eeq
Since $\zeta$ is only continuous in general, we need to find a differentiable version instead. Set 
\eq \~{\zeta}=\frac{h_{ij}\~{\eta}^{i}\~{\eta}^{j}}{g_{ij}\~{\eta}^{i}\~{\eta}^{j}}, \eeq
where $\~{\eta}=(\~{\eta}^{i})=(0,\ldots,0,1).$\ \\
At $(t_{0},\xi_{0})$ we have \eq h_{nn}=h^{n}_{n}=\k_{n}=\zeta=\~{\zeta} \eeq
and in a neighborhood of $(t_{0},\xi_{0})$ there holds
\eq \~{\zeta}\leq \zeta. \eeq
Using $h_{n}^{n}=h_{nk}g^{kn},$ we find that at $(t_{0},\xi_{0})$
\eq \dot{\~{\zeta}}=\dot{h}^{n}_{n} \eeq
and the spatial derivatives also coincide, cf. \cite[p.13]{cg:ImcfHyp}. Replacing $w$ by $\~{w}=\log\~{\zeta}+\phi+\lambda\~{u},$ we see that $\~{w}$ attains a maximal value at $(t_{0},\xi_{0}),$ where $\~{\zeta}$ satisfies the same differential equation in this point as $h^{n}_{n}.$ Thus, without loss of generality, we may pretend $h^{n}_{n}$ to be a scalar and $w$ to be given by
\eq w=\log h^{n}_{n}+\phi+\lambda\~{u}. \eeq
Since \eq \dot{\hat{\chi}}-\dot{\Phi}F^{ij}\hat{\chi}_{ij}=-\~{\chi}^{-2}(\dot{\~{\chi}}-\dot{\Phi}F^{ij}\~{\chi}_{ij})-2\~{\chi}^{-3}\dot{\Phi}F^{ij}\~{\chi}_{i}\~{\chi}_{j}, \eeq
we find \begin{align}\begin{split}
 \dot{\phi}-\dot{\Phi}F^{ij}\phi_{ij}&=(\hat{\chi}-\theta)^{-1}(\~{\chi}^{-2}(\dot{\~{\chi}}-\dot{\Phi}F^{ij}\~{\chi}_{ij})+2\~{\chi}^{-3}\dot{\Phi}F^{ij}\~{\chi}_{i}\~{\chi}_{j})\\
 						&\hphantom{=}-\dot{\Phi}F^{ij}\frac{(\hat{\chi}-\theta)_{i}(\hat{\chi}-\theta)_{j}}{(\hat{\chi}-\theta)^{2}}\\
						&=(\hat{\chi}-\theta)^{-1}(-\dot{\Phi}F^{ij}h_{ik}h^{k}_{j}\hat{\chi}+(\dot{\Phi}F+\Phi)\frac{\-{H}}{n}v\hat{\chi}+\frac{1}{n^{p}}\hat{\chi})\\
						&\hphantom{=}-\dot{\Phi}F^{ij}(\log(\hat{\chi}-\theta))_{i}(\log(\hat{\chi}-\theta))_{j}. \end{split}\end{align}
Thus, in $(t_{0},\xi_{0})$ we infer
\begin{align} \begin{split}\label{kappabound1}
0&\leq \dot{w}-\dot{\Phi}F^{ij}w_{ij}\\ 
&=\dot{\Phi}F^{kl}h_{kr}h^{r}_{l}+(\Phi-\dot{\Phi}F)h^{n}_{n}+\ddot{\Phi}F_{n}F^{n}(h^{n}_{n})^{-1}\\
					&\hphantom{=}+\dot{\Phi}F^{kl,rs}h_{kl;n}{h_{rs;}}^{n}(h^{n}_{n})^{-1}-(\Phi+\dot{\Phi}F)(h_{n}^{n})^{-1}+\dot{\Phi}F^{kl}g_{kl}\\
					&\hphantom{=}-\dot{\Phi}F^{ij}h_{ik}h^{k}_{j}\frac{\hat{\chi}}{\hat{\chi}-\theta}+(\dot{\Phi}F+\Phi)\frac{\-{H}}{n}v\frac{\hat{\chi}}{\hat{\chi}-\theta}+\frac{1}{n^{p}}\frac{\hat{\chi}}{\hat{\chi}-\theta}\\
					&\hphantom{=}+\lambda(\dot{\Phi}F-\Phi)v^{-1}-\lambda\dot{\Phi}F^{ij}\-{h}_{ij}-\frac{\lambda}{n^{p}}\\
					&\hphantom{=}-\dot{\Phi}F^{ij}(\log(\hat{\chi}-\theta))_{i}(\log(\hat{\chi}-\theta))_{j}+\dot{\Phi}F^{ij}(\log h^{n}_{n})_{i}(\log h^{n}_{n})_{j}.
					\end{split}\end{align}
In the present coordinate system we have
\eq F^{kl,rs}\eta_{kl}\eta_{rs}\leq\sum_{k\neq l}\frac{F^{kk}-F^{ll}}{\k_{k}-\k_{l}}(\eta_{kl})^{2}\leq \frac{2}{\k_{n}-\k_{1}}\sum_{k=1}^{n}(F^{nn}-F^{kk})(\eta_{nk})^{2} \eeq for all symmetric tensors $(\eta_{kl})$ and \eq F^{nn}\leq \ldots\leq F^{11},\eeq
 cf. \cite[(4.28), (4.29)]{cg:ImcfHyp} and the references therein. Using those inequalities, $\ddot{\Phi}<0$ as well as
\eq (\log h^{n}_{n})_{i}=-\phi_{i}-\lambda\~{u}_{i} \eeq in $(t_{0},\xi_{0}),$ we obtain from (\ref{kappabound1})
\begin{align} \begin{split} \label{kappabound2}
0&\leq -\dot{\Phi}F^{ij}h_{ik}h^{k}_{j}\frac{\theta}{\hat{\chi}-\theta}+(\Phi-\dot{\Phi}F)h^{n}_{n}-(\Phi+\dot{\Phi}F)(h^{n}_{n})^{-1}+\dot{\Phi}F^{kl}g_{kl}\\
   &\hphantom{=}+(\dot{\Phi}F+\Phi)\frac{\-{H}}{n}v\frac{\hat{\chi}}{\hat{\chi}-\theta}+\frac{1}{n^{p}}\frac{\hat{\chi}}{\hat{\chi}-\theta}\\
   &\hphantom{=}+\lambda(\dot{\Phi}F-\Phi)v^{-1}-\lambda\dot{\Phi}F^{ij}\-{h}_{ij}-\frac{\lambda}{n^{p}}\\
   &\hphantom{=}+2\lambda\dot{\Phi}F^{ij}\phi_{i}\~{u}_{j}+\lambda^{2}\dot{\Phi}F^{ij}\~{u}_{i}\~{u}_{j}\\
   &\hphantom{=}+\frac{2}{\k_{n}-\k_{1}}\dot{\Phi}\sum_{i=1}^{n}(F^{nn}-F^{ii})({h_{ni;}}^{n})^{2}(h^{n}_{n})^{-1}. \end{split}\end{align}
 There holds \begin{align}\begin{split} F^{ij}\-{h}_{ij}&=F^{ij}\-{g}_{ij}\coth u\geq F^{ij}\-{g}_{ij}=F^{ij}g_{ij}-F^{ij}u_{i}u_{j}\\
 									&\geq F^{ij}g_{ij}(1-\norm{Du}^{2})=v^{-2}F^{ij}g_{ij}\geq \~{c}_{0}F^{ij}g_{ij}, \end{split}\end{align}
	where $\~{c}_{0}=c(n,p,M_{0}),$ and 
	\eq h_{ni;n}=h_{nn;i}, \eeq
in view of the Codazzi equation.
We now estimate (\ref{kappabound2}).\\
We distinguish two cases.\ \\

Case 1: $\k_{1}<-\e_{1}\k_{n}, 0<\e_{1}<1.$ \ \\

Then \eq F^{ij}h_{ki}h^{k}_{j}\geq \frac 1n F^{ij}g_{ij}\e_{1}^{2}\k_{n}^{2}, \eeq
cf. \cite[p.14, (4.47)]{cg:ImcfHyp}. Furthermore, by \cite[(5.29)]{cg:spaceform}, we have
\eq v_{i}=-v^{2}h^{k}_{i}u_{k}+v\frac{\-{H}}{n}u_{i}=(-v^{2}\k_{i}+v\frac{\-{H}}{n})u_{i} \eeq
and thus \eq \norm{Dv}\leq c\abs{\k_{n}}\norm{Du}+c\norm{Du}, c=c(n,p,M_{0}) \eeq
 so that \eq \norm{D\phi}\leq c\norm{Dv}+c\norm{Du}\leq c\abs{\k_{n}}\norm{Du}+c\norm{Du}. \eeq
 Hence (\ref{kappabound2}) can be estimated:
 \begin{align} \begin{split}
 0&\leq \dot{\Phi}F^{ij}g_{ij}\Big(-\frac 1n \e_{1}^{2}\k_{n}^{2}\frac{\theta}{\hat{\chi}-\theta}+1-\lambda \~{c}_{0}+2\lambda c\norm{Du}^{2}(\k_{n}+1)\\
 	&\hphantom{=}+\lambda^{2}\norm{Du}^{2}\Big)\\
  &\hphantom{=}-(\Phi+\dot{\Phi}F)\k_{n}^{-1}+(\dot{\Phi}F+\Phi)\frac{\-{H}}{n}v\frac{\hat{\chi}}{\hat{\chi}-\theta}+\frac{1}{n^{p}}\frac{\hat{\chi}}{\hat{\chi}-\theta}\\
  &\hphantom{=}+\lambda(\dot{\Phi}F-\Phi)v^{-1}. \end{split}\end{align}
The last two lines are uniformly bounded by some $c=c(n,p,M_{0})$ and the first line converges to $-\infty,$ if $\k_{n}\rightarrow\infty,$ where we use $\dot{\Phi}F^{ij}g_{ij}\geq c>0$ and the boundedness of all the other coefficients. We conclude, that in this case any choice of $\lambda$ yields
\eq \k_{n}\leq c(n,p,M_{0}). \eeq 

Case 2: $\k_{1}\geq -\e_{1}\k_{n}.$\ \\

Then \eq \begin{aligned} &\frac{2}{\k_{n}-\k_{1}}\dot{\Phi}\sum_{i=1}^{n}(F^{nn}-F^{ii})({h_{ni;}}^{n})^{2}(h_{n}^{n})^{-1}\\ &\hphantom{=}\leq
 \frac{2}{1+\e_{1}}\dot{\Phi}\sum_{i=1}^{n}(F^{nn}-F^{ii})(\log h^{n}_{n})^{2}_{i},\end{aligned} \eeq
so that
\begin{align}\begin{split}
&\dot{\Phi}F^{ij}(\log h^{n}_{n})_{i}(\log h^{n}_{n})_{j}+\frac{2}{\k_{n}-\k_{1}}\dot{\Phi}\sum_{i=1}^{n}(F^{nn}-F^{ii})({h_{ni;}}^{n})^{2}(h^{n}_{n})^{-1}\\
&\hphantom{=}\leq \frac{2}{1+\e_{1}}\dot{\Phi}\sum_{i=1}^{n}F^{nn}(\log h^{n}_{n})^{2}_{i}-\frac{1-\e_{1}}{1+\e_{1}}\dot{\Phi}\sum_{i=1}^{n}F^{ii}(\log h^{n}_{n})_{i}^{2}\\
&\hphantom{=}\leq  \frac{2}{1+\e_{1}}\dot{\Phi}\sum_{i=1}^{n}F^{nn}(\log h^{n}_{n})^{2}_{i}-\frac{1-\e_{1}}{1+\e_{1}}\dot{\Phi}F^{nn}\sum_{i=1}^{n}(\log h^{n}_{n})^{2}_{i}\\
&\hphantom{=}=\dot{\Phi}F^{nn}\norm{D\phi+\lambda Du}^{2}\\
&\hphantom{=}=\dot{\Phi}F^{nn}(\norm{D\phi}^{2}+\lambda^{2}\norm{Du}^{2}+2\lambda\langle D\phi,D\~{u}\rangle), \end{split}\end{align}
where we used $g_{ij}=\d_{ij}.$
We now choose $\lambda=\lambda(n,p,M_{0}),$ such that
\eq \lambda >\~{c}_{0}^{-1}. \eeq
Estimating (\ref{kappabound1}) again yields
\begin{align}\begin{split} 0&\leq -\dot{\Phi}F^{nn}\k_{n}^{2}\frac{\theta}{\hat{\chi}-\theta}-(\Phi+\dot{\Phi}F)\k_{n}^{-1}+\dot{\Phi}F^{kl}g_{kl}(1-\lambda\~{c}_{0})\\
	&\hphantom{=}+(\Phi-\dot{\Phi}F)\k_{n}+(\dot{\Phi}F+\Phi)\frac{\-{H}}{n}v\frac{\hat{\chi}}{\hat{\chi}-\theta}+\frac{1}{n^{p}}\frac{\hat{\chi}}{\hat{\chi}-\theta}\\
	&\hphantom{=}+\lambda(\dot{\Phi}F-\Phi)v^{-1}-\frac{\lambda}{n^{p}}\\
 &\hphantom{=}+\dot{\Phi}F^{nn}(\lambda^{2}\norm{Du}^{2}+2\lambda\norm{D\phi}\norm{Du}), \end{split}\end{align} implying
 \eq \k_{n}(t_{0},\xi_{0})\leq c(n,p,M_{0}).\eeq
 Thus, $w$ and $\zeta$ as well, are bounded from above, implying the claim.
\end{proof}

\begin{thmss} \label{longtime}
Under the hypothesis of Theorem \ref{mainthm} we have
\eq T^{*}=\infty. \eeq
\end{thmss}

\begin{proof}
Following \cite[2.6.2]{cg:cp}, all we have to show is that we have a uniform $C^{2}(\S^{n})$ estimate on finite intervals, since we have already shown the uniform ellipticity on such intervals. There holds
\eq h^{i}_{j}=-v^{-1}\t^{-1}\~{g}^{ik}\p_{kj}+v^{-1}\frac{\dot{\t}}{\t}\d^{i}_{j}, \eeq
where $\~{g}^{ik}=\sigma^{ik}-v^{-2}\p^{i}\p^{k}.$ We have 
\eq \p_{j}=\t^{-1}u_{j} \eeq and \eq \p_{jk}=-\t^{-2}\dot{\t}u_{j}u_{k}+\t^{-1}u_{jk},\eeq
where covariant derivatives are taken with respect to $\sigma_{ij}.$
Thus \begin{align}\begin{split} \label{longtime1}
h^{i}_{j}&=\frac{\dot{\t}}{v\t}\d^{i}_{j}+v^{-1}\t^{-3}\dot{\t}\~{g}^{ik}u_{j}u_{k}-v^{-1}\t^{-2}\~{g}^{ik}u_{jk}\\
   	&=\frac{\dot{\t}}{v\t}\d^{i}_{j}+\frac{\dot{\t}}{v^{3}\t^{3}}u^{i}u_{j}-\frac{\~{g}^{ik}}{v\t^{2}}u_{kj}, \end{split}\end{align}
where $u^{i}=\sigma^{ik}u_{k}.$ Since $v\leq c,$ $\sigma_{ik}$ and $\~{g}_{ik}$ generate equivalent norms. All the other tensors are bounded in finite time and thus
\eq \abs{u}_{2,\S^{n}}\leq c=c(n,p,M_{0},T^{*}). \eeq
Then, using Krylov-Safonov, \cite{KrylovSaf}, \cite[Thm. 2.5.9]{cg:cp} and Remark \ref{shorttime} we conclude the result.
\end{proof}

\section{Decay estimates in $C^{1}$ and $C^{2}$}
\subsection*{Decay of the $C^{1}$-norm}

\begin{thmss} \label{GradDecay}
Under the hypotheses of Theorem \ref{mainthm}, for all $0<p\leq 1$ there exist constants $0<\lambda$ and $0<c$ depending on $n,p$ and $M_{0},$ such that
\eq v-1\leq ce^{-\lambda t}\ \ \forall t\in[0,\infty). \eeq
In case $p>1$ there exist constants $0<\e,\lambda,c,$ depending on $n,p$ and $M_{0},$ such that
\eq  \osc u(0,\cdot)<\e \Rightarrow\  v-1\leq ce^{-\lambda t}\ \ \forall t\in[0,\infty).\eeq
 \end{thmss}

\begin{proof}
Considering the equation for $v,$ cf. \cite[(5.28)]{cg:spaceform} and, using \eq\frac{\dot{\-{H}}}{n}=1-\frac{\-{H}^{2}}{n^{2}},\eeq we obtain
\begin{align}\begin{split}
\dot{v}-\dot{\Phi}F^{ij}v_{ij}&=-\dot{\Phi}F^{ij}h_{ik}h^{k}_{j}v-2v^{-1}\dot{\Phi}F^{ij}v_{i}v_{j}+2\dot{\Phi}F^{ij}v_{i}u_{j}\frac{\-{H}}{n}\\
					&\hphantom{=}-\dot{\Phi}F^{ij}g_{ij}\frac{\-{H}^{2}}{n^{2}}v-\dot{\Phi}F^{ij}u_{i}u_{j}v+\dot{\Phi}F^{ij}u_{i}u_{j}\frac{\-{H}^{2}}{n^{2}}v\\
					&\hphantom{=}+\frac{\-{H}}{n}(v^{2}-1)(\Phi-\dot{\Phi}F)+2\dot{\Phi}F\frac{\-{H}}{n}v^{2}\\
					&=-\dot{\Phi}F^{ij}(h_{ik}h^{k}_{j}-2h_{ij}+g_{ij})v-2v^{-1}\dot{\Phi}F^{ij}v_{i}v_{j}\\
					&\hphantom{=}+2\dot{\Phi}F^{ij}v_{i}u_{j}\frac{\-{H}}{n}-\dot{\Phi}F^{ij}g_{ij}\left(\frac{\-{H}^{2}}{n^{2}}-1\right)v\\
					&\hphantom{=}+\dot{\Phi}F^{ij}u_{i}u_{j}\left(\frac{\-{H}^{2}}{n^{2}}-1\right)v+\frac{\-{H}}{n}(v^{2}-1)\Phi\\
					&\hphantom{=}+\dot{\Phi}F\frac{\-{H}}{n}-2\dot{\Phi}Fv+\dot{\Phi}F\frac{\-{H}}{n}v^{2} \end{split}\end{align}

Let $\lambda>0$ and set \eq w=(v-1)e^{\lambda t}.\eeq
Fix $T>0$ and suppose \eq\sup_{[0,T]\times M}w=w(t_{0},\xi_{0})>1.\eeq Then at this point 
\begin{align}\begin{split} \label{GradDecay2}
0&\leq \dot{\Phi} F^{ij}u_{i}u_{j}\left(\frac{\-{H}^{2}}{n^{2}}-1\right)ve^{\lambda t}+\frac{\-{H}}{n}(v^{2}-1)\Phi e^{\lambda t}+\dot{\Phi} F\frac{\-{H}}{n}(v-1)^{2}e^{\lambda t}\\
  &\hphantom{=}+2\dot{\Phi} F\left(\frac{\-{H}}{n}-1\right)ve^{\lambda t}+\lambda w\\
  &=\dot{\Phi} F^{ij}u_{i}u_{j}\left(\frac{\-{H}^{2}}{n^{2}}-1\right)ve^{\lambda t}+2\dot{\Phi} F\left(\frac{\-{H}}{n}-1\right)ve^{\lambda t}\\
  &\hphantom{=}+\left(\frac{\-{H}}{n}F^{-p}(p(v-1)-(v+1))+\lambda\right)w.\\
  &\leq ce^{(\lambda-\frac{2}{n^{p}})t}+\left(\frac{\-{H}}{n}F^{-p}(p(v-1)-(v+1))+\lambda\right)w,
  \end{split}\end{align}		
  where the last estimate follows from the estimates of the curvature function, the principal curvatures and Corollary \ref{thetaGrowth}. The constant in this inequality depends on $n,p$ and $M_{0}.$\ \\
  
Consider $p>1.$ 
In view of (\ref{convexitygradbound}) we deduce

\eq v\leq e^{\-{\k}\osc u},\eeq
where $\-{\k}$ is an upper bound for the curvatures of the slices, which in our case converge to $1,$ as $t\rightarrow \infty.$ 
Choosing $\b>0,$ such that
\eq \b<\frac{1}{\-{\k}}\log\frac{p+1}{p-1}\ \ \forall t\in[0,\infty), \eeq
there exists $\e>0,$ such that
\eq \osc u(0,\cdot)<\e \Rightarrow \sup_{t\in[0,\infty)}\osc u(t,\cdot)<\b,\eeq
due to the estimates (\ref{sphFlowBound1}) and (\ref{oscBound1}) and we conclude further
 \eq v\leq e^{\-{\k}\b}<\frac{p+1}{p-1}\ \ \forall t\in[0,\infty).  \eeq
Using \eq 0<c^{-1}\leq F\leq c,\  c=c(n,p,M_{0}) \eeq and \eq \frac{\bar{H}}{n}\geq 1,\eeq we obtain from
(\ref{GradDecay2})
\eq 0\leq ce^{(\lambda-\frac{2}{n^{p}}t)}+c((p-1)v-(p+1)+\lambda)w.\eeq 

In this inequality the coefficient of the linear term is strictly negative in view of the previous considerations, if $\lambda(n,p,M_{0})$ is small, while the first term converges to $0,$ which leads to a contradiction, if $t_{0}$ is sufficiently large. Thus $w$ is bounded, completing the proof.

\end{proof}

\subsection*{Curvature asymptotics}

\begin{lemss} \label{Lipschitz}
Let $f\in C^{0,1}(\R_{+})$ and let $D$ be the set of points of differentiability of $f.$ Suppose that for all
$\e>0$ there exist $T_{\e}>0$ and $\d_{\e}>0,$ such that \eq \{t\in D \cap  [T_{\e},\infty)\colon f(t)\geq \e\}\subset \{t\in D \cap [T_{\e},\infty)\colon f'(t) <-\d_{\e}\}. \eeq
Then there holds
\eq \limsup\limits_{t\rightarrow\infty}f(t)\leq 0. \eeq 
\end{lemss}

\begin{proof}
Suppose first, that \eq \liminf\limits_{t\rightarrow\infty}f(t)\geq 2\e>0.\eeq Then there exists $\~{T}>0,$ such that
\eq f(t)>\e\ \ \forall t\geq\~{T} \eeq and hence, there exists $T_{\e}\geq \~{T}$ and $\d_{\e}>0,$ such that
\eq  f'(t)<-\d_{\e}\ \ \forall t\in D\cap[T_{\e},\infty) \eeq and we infer for all $t\geq T_{\e}$
\eq f(t)\leq f(T_{\e})+\int_{T_{\e}}^{t}(-\d_{\e})=f(T_{\e})-\d_{\e}(t-T_{\e}) \rightarrow -\infty, \eeq
as $t\rightarrow \infty,$ which is a contradiction.\ \\

Now suppose that
\eq \liminf\limits_{t\rightarrow\infty}f(t)\leq 0\ \wedge\ \limsup\limits_{t\rightarrow\infty}f(t)\geq 2\e>0. \eeq
Then there exist $(t_{k})_{k\in\N}$ and $(s_{k})_{k\in\N},$ such that
\begin{align}\begin{split} &t_{k}<s_{k},\\
		&t_{k}\rightarrow\infty, k\rightarrow\infty,\\
		&\frac{\e}{2}\leq f(t_{k})\leq \e,\\
		&f(s_{k})>\frac 32 \e,\\\
		&f_{|[t_{k},s_{k})}\geq \frac{\e}{2}. \end{split} \end{align}
Since $D\subset\R_{+}$ is dense and $f$ continuous, we may suppose that $t_{k},s_{k}\in D.$
Choose $T_{\frac{\e}{2}}>0$ and $\d_{\frac{\e}{2}}>0,$ such that
\eq  f(t)\geq\frac{\e}{2}\Rightarrow f'(t)<-\d_{\frac{\e}{2}}\ \ \forall t\in D \cap [T_{\frac{\e}{2}},\infty). \eeq We conclude
\eq f(s_{k})-f(t_{k})\leq -\int_{t_{k}}^{s_{k}}\d_{\frac{\e}{2}}=-\d_{\frac{\e}{2}}(s_{k}-t_{k})\ \ \forall t_{k}, s_{k}\geq T_{\frac{\e}{2}} , \eeq hence
\eq f(s_{k})<f(t_{k}), \eeq
which is a contradiction.
\end{proof}
\vspace{0,01 cm}

\begin{lemss} \label{kappaconv}
Under the hypotheses of Theorem \ref{mainthm} the principal curvatures of the flow hypersurfaces converge to $1,$
\eq \sup\limits_{M}\abs{\k_{i}(t,\cdot)-1}\rightarrow 0,\ t\rightarrow \infty,\ \ \forall 1\leq i\leq n. \eeq
\end{lemss}

\begin{proof}
(i) As in the proof of Proposition \ref{kappaBound} we consider the function
\eq \zeta=\sup\{h_{ij}\eta^{i}\eta^{j}\colon \norm{\eta}^{2}=g_{ij}\eta^{i}\eta^{j}=1\}. \eeq
Set \eq w=(\log \zeta+\log \~{\chi}+\~{u}-\log 2)t, \eeq
where \eq \~{\chi}=\chi e^{\frac{t}{n^{p}}}=\frac{v}{\sinh u}e^{\frac{t}{n^{p}}},\ \~{u}=u-\frac{t}{n^{p}}. \eeq
Fix $0<T<\infty$ and suppose \eq \sup_{[0,T]\times M}w=w(t_{0},\xi_{0}),\ t_{0}>0.\eeq
As in the proof of Proposition \ref{kappaBound}, we choose coordinates such that in $(t_{0},\xi_{0})$ there holds $g_{ij}=\d_{ij},$ $h_{ij}=\k_{i}\d_{ij}$ and
\eq w=(\log h^{n}_{n}+\log \~{\chi}+\~{u}-\log 2)t. \eeq
 First note, that \begin{align}\begin{split} (\log \~{\chi}+\~{u}-\log 2)t&=\left(\log v-\log(\sinh u)+\frac{t}{n^{p}}+u-\frac{t}{n^{p}}-\log 2\right)t \\
												&=\left(\log v-\log \frac 12(e^{u}-e^{-u})+u-\log 2\right)t\\
												&=(\log v-\log (e^{u}-e^{-u})+u)t \end{split}\end{align}
is bounded. To prove this claim, note that \eq t\log v=\log(1+v-1)^{t}\leq \log(1+ce^{-\lambda t})^{t}, \eeq which follows from Theorem \ref{GradDecay}. Furthermore
\eq e^{t(u-\log(e^{u}-e^{-u}))}=\left( \frac{e^{u}}{e^{u}-e^{-u}}\right)^{t}=(1-e^{-2u})^{-t}\leq (1-e^{c-\frac{t}{n^{p}}})^{-t},\eeq following from Corollary \ref{sphFlowGrowth} and Lemma \ref{oscBound}. But for large $t$ we have \eq e^{-\lambda t}\leq \frac{c}{t} \eeq and thus \eq (1+ce^{-\lambda t})^{t}\leq (1+\frac{c}{t})^{t}\leq \const. \eeq The term \eq t(u-\log(e^{u}-e^{-u}))\eeq is bounded for the same reason.
Using the equations for $h^{n}_{n},$ $\~{\chi}$ and $\~{u},$ cf. Proposition \ref{kappaBound},  we obtain
\begin{align} \begin{split}
\dot{w}-\dot{\Phi}F^{ij}w_{ij}&=\Big((\Phi-\dot{\Phi}F)h^{kn}h_{kn}(h^{n}_{n})^{-1}+\ddot{\Phi}F_{n}F^{n}(h^{n}_{n})^{-1}\\
					&\hphantom{=}+\dot{\Phi}F^{kl,rs}h_{kl;n}{h_{rs;}}^{n}(h^{n}_{n})^{-1}-(\Phi+\dot{\Phi}F)(h^{n}_{n})^{-1}\\ 
					&\hphantom{=}+\dot{\Phi}F^{kl}g_{kl}+\dot{\Phi}F^{kl}(\log h^{n}_{n})_{k}(\log h^{n}_{n})_{l}\\
					&\hphantom{=}-\dot{\Phi}F^{kl}(\log\~{\chi})_{k}(\log\~{\chi})_{l}+(\dot{\Phi}F+\Phi)\frac{\-{H}}{n}v\\
					&\hphantom{=}+(\dot{\Phi}F-\Phi)v^{-1}-\dot{\Phi}F^{ij}\-{h}_{ij}\Big)t_{0}\\
					&\hphantom{=}+(\log h^{n}_{n}+\log\~{\chi}+\~{u}-\log 2)\\
					&\leq \Phi\left(h^{n}_{n}-(h^{n}_{n})^{-1}-v^{-1}+\frac{\-{H}}{n}v\right)t_{0}\\
					&\hphantom{=}+\dot{\Phi}F\left(\frac{\-{H}}{n}v+v^{-1}-(h^{n}_{n}+(h^{n}_{n})^{-1})\right)t_{0} \\
					&\hphantom{=}+\dot{\Phi}F^{kl}\-{g}_{kl}(1-\coth u)t_{0}+\dot{\Phi}F^{kl}u_{k}u_{l}t_{0}\\
					&\hphantom{=}+\dot{\Phi}F^{kl}((\log h^{n}_{n})_{k}(\log h^{n}_{n})_{l}-(\log\~{\chi})_{k}(\log\~{\chi})_{l})t_{0}\\
					&\hphantom{=}+\log h^{n}_{n}+\log\~{\chi}+\~{u}-\log 2.
					 \end{split}\end{align}
At $(t_{0},\xi_{0})$ we have
\eq (\log h^{n}_{n})_{k}=-(\log \~{\chi})_{k}-u_{k} \eeq
and thus
\begin{align}\begin{split}\label{kappaconv2}
0&\leq \Phi\left(h^{n}_{n}-(h^{n}_{n})^{-1}-v^{-1}+\frac{\-{H}}{n}v\right)t_{0}\\
 &\hphantom{=}+\dot{\Phi}F\left(\frac{\-{H}}{n}v+v^{-1}-(h^{n}_{n}+(h^{n}_{n})^{-1})\right)t_{0}\\
		&\hphantom{=}+\log h^{n}_{n}+\log \~{\chi}+\~{u}-\log 2+2\dot{\Phi}F^{kl}u_{k}u_{l}t_{0}+2\dot{\Phi}F^{kl}(\log\~{\chi})_{k}u_{l}t_{0} \end{split}\end{align}		
We have \eq (\log\~{\chi})_{k}=\frac{\chi_{k}}{\chi}=\frac{\sinh u}{v}\frac{v_{k}\sinh u-vu_{k}\cosh u }{\sinh^{2}u}\rightarrow 0, \eeq
since \eq v_{k}=-v^{2}h^{i}_{k}u_{i}+v\frac{\-{H}}{n}u_{k},\eeq the principal curvatures are bounded by Proposition \ref{kappaBound} and $|Du|^{2}\rightarrow 0$ by Theorem \ref{GradDecay}.
In view of \eq x+x^{-1}\geq 2\ \ \forall x>0 \eeq
and by Theorem \ref{GradDecay} we have in $(t_{0},\xi_{0}):$
\eq 0\leq \Phi(h^{n}_{n}-(h^{n}_{n})^{-1})t_{0}+c \eeq
for some $c=c(n,p,M_{0}),$ which implies
\eq h^{n}_{n}-(h^{n}_{n})^{-1}\leq \frac{cF^{p}}{t_{0}}.\eeq
Thus we find
\begin{align}\begin{split}
w&\leq t_{0}\log\left(1+\frac{cF^{p}}{t_{0}}\right)+t_{0}(\log\~{\chi}+\~{u}-\log 2)\\
	&\leq c=c(n,p,M_{0}). \end{split}\end{align}
Hence $w$ is a priori bounded and thus
\eq \limsup\limits_{t\rightarrow\infty}\sup\limits_{M} \k_{i}(t,\cdot)\leq 1\ \ \forall 1\leq i\leq n. \eeq

(ii) Now we investigate the function
\eq z=\log(-\Phi)+\log\~{\chi}+\~{u}-\log 2-\log\frac{1}{n^{p}}\eeq
and show that \eq \limsup\limits_{t\rightarrow\infty}\sup\limits_{M} z(t,\cdot)\leq 0.\eeq
The Lipschitz function \eq \~{z}=\sup\limits_{\xi\in M}z(\cdot,\xi) \eeq satisfies for almost every $t\geq 0$
\begin{align}\begin{split} \dot{\~{z}}&\leq \dot{\Phi}F^{kl}((\log(-\Phi))_{k}(\log(-\Phi))_{l}-(\log\~{\chi})_{k}(\log\~{\chi})_{l})\\
							&\hphantom{=}+\Phi\left(\frac{\-{H}}{n}v-v^{-1}\right)+\dot{\Phi}\left(F\frac{\-{H}}{n}v-F^{kl}\-{h}_{kl}\right)\\
							&\hphantom{=}+\dot{\Phi}(Fv^{-1}-F^{kl}g_{kl})\\
							&\leq o(1)+\dot{\Phi}\left(F\frac{\-{H}}{n}v+Fv^{-1}-2F^{kl}g_{kl}\right). \end{split}\end{align}
\textbf{Claim}:
$ \forall\e>0\ \exists T>0\ \exists\d>0\colon$  $$A_{\e}=\{t\in[T,\infty) \cap  D\colon\~{z}(t)>\e\}\subset\{t\in[T,\infty) \cap D\colon \dot{\~{z}}(t)\leq -\d\}, $$ where $D$ is the set of points of differentiability of $\~{z}.$
\ \\

To prove this claim, let $\e>0$ and choose $T>0$ such that \eq \log\~{\chi}+\~{u}-\log 2<\frac{\e}{2}\ \ \forall (t,\xi)\in[T,\infty)\times M. \eeq
Then for $t\in A_{\e}$ we have
\eq \left(\log(-\Phi)-\log\frac{1}{n^{p}}\right)(t,\xi_{t})>\frac{\e}{2}, \eeq
where $\~{z}(t)=z(t,\xi_{t}).$
Thus there exists $0<\g=\g(\e),$ such that \eq F(t,\xi_{t})<n-\g,\eeq
 implying \eq F\frac{\-{H}}{n}v+Fv^{-1}-2F^{kl}g_{kl}\leq \-{H}v-n-\frac{\-{H}}{n}v\g.\eeq
 One may enlarge $T$, such that
 \eq  \abs{o(1)+\dot{\Phi}(\-{H}v-n)}\leq \frac{(\inf \dot{\Phi})\g}{2}\ \ \forall (t,\xi)\in[T,\infty)\times M. \eeq
Thus \eq \dot{\~{z}}(t)\leq -(\inf \dot{\Phi})\frac{\g}{2}=:-\d. \eeq
Now it follows from Lemma \ref{Lipschitz}, that \eq \limsup\limits_{t\rightarrow\infty}\~{z}(t)\leq 0.\eeq
Thus
\begin{align}\begin{split}
 &\limsup\limits_{t\rightarrow\infty}\sup\limits_{M} \log(-\Phi)-\log\frac{1}{n^{p}}=\limsup\limits_{t\rightarrow\infty}\sup\limits_{M} (z-\log\~{\chi}-\~{u}+\log 2)\\
    		&\hphantom{=}\leq \limsup\limits_{t\rightarrow\infty}\~{z}+\limsup\limits_{t\rightarrow\infty}\sup\limits_{M}(-\log\~{\chi}-\~{u}+\log 2)\\
		&\hphantom{=}\leq 0, \end{split}\end{align}
implying
\begin{align}\begin{split}
\frac{1}{n^{p}}&\geq \limsup\limits_{t\rightarrow\infty}\sup\limits_{M}\frac{1}{F^{p}}=\limsup\limits_{t\rightarrow\infty}(\inf\limits_{M}F^{p})^{-1}\\
			&=(\liminf\limits_{t\rightarrow\infty}\inf\limits_{M}F^{p})^{-1}.\end{split}\end{align}
			This leads to
\eq \liminf\limits_{t\rightarrow\infty}\inf\limits_{M}F^{p}\geq n^{p}\eeq
Together with (i) we obtain
\eq \sup\limits_{M}\abs{F-n}\rightarrow 0.\eeq
Now suppose there was a sequence $(t_{k},\xi_{k})$ such that for the smallest eigenvalue we had
\eq \k_{1}(t_{k},\xi_{k})\rightarrow\d<1.\eeq
Then\begin{align}\begin{split}
\limsup\limits_{k\rightarrow\infty}F(\k_{1},\ldots,\k_{n})-n&=\limsup\limits_{k\rightarrow\infty}\sum_{i=1}^{n}\frac{\del F}{\del\k_{i}}(\~{\k}_{k})(\k^{i}-1)\\
											&\leq \limsup\limits_{k\rightarrow\infty}F^{1}(\d-1)<0, \end{split}\end{align}
				which is a contradiction.
\end{proof}

\subsection*{Optimal rates of convergence}

We now derive the optimal speed of convergence of the second fundamental form to $\d^{i}_{j},$ which, of course, can not be better than what we expect from the spherical flow, i.e.
\begin{align}\begin{split} \abs{\-{h}_{ij}-\-{g}_{ij}}&=\abs{\coth\Theta-1}\abs{\delta^{i}_{j}}\\
									&\leq c\left|\frac{\cosh\Theta-\sinh\Theta}{\sinh\Theta}\right|\\
									&\leq c\frac{e^{-\Theta}}{e^{\Theta}-e^{-\Theta}}\\
									&\leq ce^{-2\Theta}\leq ce^{-\frac{2}{n^{p}}t}. \end{split}\end{align} 

\begin{thmss} \label{kappaDecay}
The principal curvatures of the flow hypersurfaces of (\ref{floweq}) converge to $1$ exponentially fast. There exists $c=c(n,p,M_{0}),$ such that
\eq \abs{h^{i}_{j}-\d^{i}_{j}}\leq ce^{-\frac{2}{n^{p}}t}\ \ \forall(t,\xi)\in[0,\infty)\times M.\eeq
\end{thmss}									

\begin{proof}
Also compare \cite[Thm. 5.1]{ding:hyperbolic}, where the author uses the same function $G.$\\

(i) Define \eq G=\frac 12 \abs{h^{i}_{j}-\d^{i}_{j}}^{2}e^{\lambda t}=\frac 12(h^{i}_{j}-\d^{i}_{j})(h^{j}_{i}-\d^{j}_{i})e^{\lambda t}, \lambda>0.\eeq
Then \begin{align}\begin{split}
\dot{G}-\dot{\Phi}F^{kl}G_{kl}&=\Big((\dot{h}^{i}_{j}-\dot{\Phi}F^{kl}h^{i}_{j;kl})(h^{j}_{i}-\d^{j}_{i})-\dot{\Phi}F^{kl}h^{i}_{j;k}h^{j}_{i;l}\Big)e^{\lambda t}+\lambda G\\
				&=\Big(\dot{\Phi}F^{kl}h_{kr}h^{r}_{l}h^{i}_{j}(h^{j}_{i}-\d^{j}_{i})+(\Phi-\dot{\Phi}F)h^{ki}h_{kj}(h^{j}_{i}-\d^{j}_{i})\\
				&\hphantom{=}+\ddot{\Phi}F_{j}F^{i}(h^{j}_{i}-\d^{j}_{i})+\dot{\Phi}F^{kl,rs}h_{kl;j}{h_{rs;}}^{i}(h^{j}_{i}-\d^{j}_{i})\\
				&\hphantom{=}-(\Phi+\dot{\Phi}F)\d^{i}_{j}(h^{j}_{i}-\d^{j}_{i})+\dot{\Phi}F^{kl}g_{kl}h^{i}_{j}(h^{j}_{i}-\d^{j}_{i})\\
				&\hphantom{=}-\dot{\Phi}F^{kl}h^{i}_{j;k}h^{j}_{i;l}\Big)e^{\lambda t}+\lambda G. \end{split}\end{align}
Fix $0<T<\infty$ and suppose \eq \sup_{[0,T]\times M}G=G(t_{0},\xi_{0})>0,\ t_{0}>0.\eeq

Since $\abs{h^{i}_{j}-\d^{i}_{j}}\rightarrow 0$ using Lemma \ref{kappaconv}, we may suppose that $t_{0}$ is so large, that bad terms involving derivatives of the second fundamental form can be absorbed by the term $-\dot{\Phi}F^{kl}h^{i}_{j;k}h^{j}_{i;l}.$
There holds
\eq h^{i}_{k}h^{k}_{j}=(h^{i}_{k}-\d^{i}_{k})(h^{k}_{j}-\d^{k}_{j})+2(h^{i}_{j}-\d^{i}_{j})+\d^{i}_{j}.\eeq
Thus there exists $T_{0}=T_{0}(n,p,M_{0}),$ such that for $t_{0}>T_{0}$ we have
\begin{align}\begin{split}
0&\leq \Big(\dot{\Phi}F^{kl}h_{kr}h^{r}_{l}h^{i}_{j}(h^{j}_{i}-\d^{j}_{i})+(\Phi-\dot{\Phi}F)(h^{i}_{k}-\d^{i}_{k})(h^{k}_{j}-\d^{k}_{j})(h^{j}_{i}-\d^{j}_{i})\\
	&\hphantom{=}+2(\Phi-\dot{\Phi}F)(h^{i}_{j}-\d^{i}_{j})(h^{j}_{i}-\d^{j}_{i})+(\Phi-\dot{\Phi}F)\d^{i}_{j}(h^{j}_{i}-\d^{j}_{i})\\
	&\hphantom{=}-(\Phi+\dot{\Phi}F)\d^{i}_{j}(h^{j}_{i}-\d^{j}_{i})+\dot{\Phi}F^{kl}g_{kl}h^{i}_{j}(h^{j}_{i}-\d^{j}_{i})\Big)e^{\lambda t_{0}}+\lambda G\\
	&=\Big(\dot{\Phi}F^{kl}(h_{kr}h^{r}_{l}-2h_{kl}+g_{kl})h^{i}_{j}(h^{j}_{i}-\d^{j}_{i})\\
	&\hphantom{=}+(\Phi-\dot{\Phi}F)(h^{i}_{k}-\d^{i}_{k})(h^{k}_{j}-\d^{k}_{j})(h^{j}_{i}-\d^{j}_{i})\\ 
	&\hphantom{=}+2\Phi(h^{i}_{j}-\d^{i}_{j})(h^{j}_{i}-\d^{j}_{i})\Big)e^{\lambda t_{0}}+\lambda G. \end{split}\end{align}
In $(t_{0},\xi_{0})$ choose coordinates, such that 
\eq g_{ij}=\d_{ij},\ h_{ij}=\k_{i}\d_{ij},\ \k_{1}\leq \ldots\leq \k_{n}.\eeq

Then
\begin{align}\begin{split} \label{kappaDecay1} 0&\leq \Big(\dot{\Phi}F^{ii}(\k_{i}-1)^{2}\sum_{j=1}^{n}\k_{j}(\k_{j}-1)\\
			&\hphantom{=}+(\Phi-\dot{\Phi}F)\sum_{i=1}^{n}(\k_{i}-1)^{3}\Big)e^{\lambda t_{0}}+(4\Phi+\lambda)G \\
					&\leq \Big(-4F^{-p}+\lambda +2\dot{\Phi}\sum_{j=1}^{n}\abs{\k_{j}}\abs{\k_{j}-1}\sum_{m=1}^{n}F^{mm}\\
								&\hphantom{=}+2\abs{\Phi-\dot{\Phi}F}\max\limits_{1\leq j\leq n}\abs{\k_{j}-1}\Big)G. \end{split}\end{align}
Enlarging $T_{0},$ we obtain a contradiction, if $\lambda>0$ is small.\ \\

(ii) By Proposition \ref{kappaBound} we know that $\Phi=\Phi(F(\kappa_{i}))$ is uniformly Lipschitz continuous with respect to $(\kappa_{i})$ during the flow. 
Thus \eq \abs{-4F^{-p}+\frac{4}{n^{p}}}\leq c\max\limits_{i}\abs{\kappa_{i}-1}\leq ce^{-\frac{\lambda}{2}t}.\eeq
Now define \eq \~{G}=\sup\limits_{M}\frac 12\abs{h^{i}_{j}-\d^{i}_{j}}^{2}e^{\frac{4}{n^{p}}t}.\eeq
Then for $t\geq T_{0},$ cf. (i), we obtain from (\ref{kappaDecay1})
\begin{align} \begin{split} \dot{\~{G}}&\leq \Big(-4F^{-p}+\frac{4}{n^{p}} +2\dot{\Phi}\sum_{j=1}^{n}\abs{\k_{j}}\abs{\k_{j}-1}\sum_{m=1}^{n}F^{mm}\\
								&\hphantom{=}+2\abs{\Phi-\dot{\Phi}F}\max\limits_{1\leq j\leq n}\abs{\k_{j}-1}\Big)\~{G}\\
								&\leq ce^{-\frac{\lambda}{2}t}\~{G},\ c=c(n,p,M_{0}),\ \lambda=\lambda(n,p,M_{0}).\end{split} \end{align}
Thus \eq \~{G}\leq c(n,p,M_{0}),\eeq
which implies the claim.
\end{proof}

\begin{thmss} \label{optGradDecay}
In both cases of Theorem \ref{GradDecay} the conclusions hold with $\lambda=\frac{2}{n^{p}}.$
\end{thmss}

\begin{proof}
We come back to the proof of Lemma \ref{GradBound} and define
\eq \~{w}=\sup\limits_{x\in\S^{n}}\frac 12 \abs{D\p(\cdot,x)}^{2}=w(t,x_{t}).\eeq
Using the same calculation as in (\ref{GradBound1}), we obtain
\begin{align}\begin{split}
\dot{\~{w}}&\leq \big(2(p-1)F^{-p}v\t^{p-1}\dot{\t}-2pF^{-(p+1)}\t^{p+1}F^{kl}\~{g}_{kl}\big)\~{w}\\
		&\leq \big(2p(F^{-p}v\t^{p-1}\dot{\t}-F^{-(p+1)}\t^{p+1}F^{kl}\~{g}_{kl})-2F^{-p}v\t^{p}\big)\~{w}, \end{split}\end{align}
		where $F=F(\~{h}^{i}_{j})=F(\t h^{i}_{j}).$\ \\
		
We have \eq \abs{v\dot{\t}-\t}\leq \abs{v\dot{\t}-\dot{\t}}+\abs{\dot{\t}-\t}\leq c\dot{\t}e^{-\lambda t}+e^{-\frac{t}{n^{p}}},	\eeq
and thus 
\eq\dot{\~{w}}\leq (2pF^{-p}(v\dot{\t}-\t)\t^{-1}+2p(F^{-p}-F^{-(p+1)}n)-2F^{-p})\~{w}, \eeq
where now $F=F(h^{i}_{j}).$
Since \eq \abs{F-n}\leq ce^{-\frac{\lambda}{2}t},\eeq	
we obtain \eq \dot{\~{w}}\leq \left(ce^{-\frac{\lambda}{2} t}-\frac{2}{n^{p}}\right)\~{w}, \eeq
which implies \eq \~{w}\leq ce^{-\frac{2}{n^{p}}t}.\eeq		
\end{proof}

\begin{thmss}\label{phi2DerDecay}
For the function $\p$ in Remark \ref{phi} there exists $c=c(n,p,M_{0}),$ such that
\eq \abs{D^{2}\p}\leq ce^{-\frac{t}{n^{p}}}, \eeq
where the derivatives as well as the norm are taken with respect to $\sigma_{ij}.$
\end{thmss}

\begin{proof}
(\ref{phiA}) implies
\eq \p^{i}_{j}=\sigma^{ik}\p_{kj}=v^{-2}\p^{i}\p^{k}\p_{kj}+\dot{\t}\d^{i}_{j}-v\t h^{i}_{j}. \eeq
In view of Theorem \ref{kappaDecay} and Theorem \ref{optGradDecay} we deduce \eq|\dot{\t}-\t|=e^{-u}\leq ce^{-\frac{t}{n^{p}}}\eeq and, using (\ref{thetaGrowth1}), 
we obtain
\begin{align}\begin{split}
\abs{D^{2}\p}&\leq \abs{v^{-2}\p^{i}\p^{k}\p_{kj}}+\abs{\dot{\t}\d^{i}_{j}-v\t h^{i}_{j}}\\
		&\leq c\abs{D\p}^{2}\abs{D^{2}\p}+\abs{\dot{\t}\d^{i}_{j}-\t\d^{i}_{j}}+\abs{\t\d^{i}_{j}-v\t\d^{i}_{j}}+\abs{v\t\d^{i}_{j}-v\t h^{i}_{j}}\\
		&\leq c\abs{D\p}^{2}\abs{D^{2}\p}+ce^{-\frac{t}{n^{p}}}\\
		&\leq \~{c}e^{-\frac{2}{n^{p}}t}\abs{D^{2}\p}+ce^{-\frac{t}{n^{p}}}, \end{split}\end{align}
where $c,\~{c}$ depend on $n,p$ and $M_{0}.$ Choosing $T=T(n,p,M_{0})$ such that 
\eq \~{c}e^{-\frac{2}{n^{p}}t}<\frac 12\ \ \forall t\geq T. \eeq
we obtain the claim.
\end{proof}

\section{Decay estimates of higher order}

We first need a definition to simplify the notation, compare \cite[Def. 6.6]{cg:ImcfHyp}, and the remark afterwards.

\begin{definition}
(1) For tensors $S$ and $T$, the symbol $S\star T$ denotes an arbitrary linear combination of contractions of $S\otimes T.$ We do not distinguish between $S\star T$ and $cS\star T,$ $c=c(n,p,M_{0}).$\ \\

(2) For $\e\in\R,$ the symbol $\O^{\e}$ denotes an arbitrary tensor, which can be estimated like
\eq \abs{\O^{\e}}\leq c_{\e}e^{\frac{\e t}{n^{p}}}, c_{\e}=c(n,p,M_{0},\e), \eeq
where the norm is taken with respect to the spherical metric.\ \\

(3) For a tensor $T,$ the symbol $D^{k}T$ denotes an arbitrary covariant derivative of order $k$ with respect to the spherical metric.\ \\

(4) If a derivative of order $m$ is expressed as an algebraic combination of terms involving $\O^{\e},$ then the corresponding constants may additionally depend on $m.$
\end{definition}

Until now we have shown, that the function \eq \p=\int_{r_{0}}^{u}\t^{-1} \eeq
 satisfies the scalar parabolic equation \eq \label{phiScalarFlow5}\frac{\del\p}{\del t}\equiv\dot{\p}=-\t^{p-1}v\Phi\ \mathrm{on}\ [0,\infty)\times\S^{n}, \eeq
 where $F=F(\~{h}^{i}_{j})=F(\t h^{i}_{j}).$
 Furthermore, we have proven the estimates
 \eq D\p=\O^{-1}\eeq and
 \eq D^{2}\p=\O^{-1}.\eeq
 On the following pages we prove analogous estimates for higher derivatives of $\p$ by differentiating (\ref{phiScalarFlow}). We prepare the final result by examining all of the terms separately first. In the sequel, we suppose $m\geq 3.$
 
 \begin{lemss} \label{derivrules}
 For functions $g,f^{i}\colon M\rightarrow \R$ on a manifold the following generalizations of the product- and chain rule hold for higher derivatives. 
 \eq  D^{m}\left(\prod_{i=1}^{k}f^{i}\right)=\sum_{j_{1}+\ldots +j_{k}=m}c_{m,k}\prod_{i=1}^{k}D^{j_{i}}f^{i}, \eeq
 \eq \label{diBruno} D^{m}(f\circ g)=\sum_{k_{1}+\ldots +mk_{m}=m}\frac{m!}{k_{1}!\dots k_{m}!}D^{\sum_{i=1}^{m}k_{i}}f(g)\prod_{i=1}^{m}\left(\frac{D^{i}g}{i!}\right)^{k_{i}}. \eeq
 \end{lemss}
 
 \begin{proof}
 For $m=1$ this is the ordinary product rule. If the claim holds for $m\geq 1,$ we find
 \begin{align} \begin{split}
 D^{m+1}\left(\prod_{i=1}^{k}f^{i}\right)&=D^{m}\left(\sum_{j_{1}+\ldots+j_{k}=1}\prod_{i=1}^{k}D^{j_{i}}f^{i}\right)\\
 						&=\sum_{j_{1}+\ldots +j_{k}=1}\sum_{l_{1}+\ldots+l_{k}=m}c_{m,k}\prod_{i=1}^{k}D^{l_{i}+j_{i}}f^{i}\\
						&=\sum_{(j_{1}+l_{1})+\ldots (j_{k}+l_{k})=m+1}\~{c}_{m,k}\prod_{i=1}^{k}D^{j_{i}+l_{i}}f^{i}, \end{split}\end{align}
	as desired.

The generalized chain rule is known as formula of Fa\'a di Bruno, cf. \cite[p. 17, Thm. 1.3.2]{KrPa}.  
 \end{proof}
 
 Let us remark, that the cited version of the generalized chain rule is the one, which holds for functions depending on one variable. Although our functions depend on $n$ variables, all that matters is the order of the multiindex in most of the cases, so that such a formal version is all we need.
 
 \begin{lemss} \label{higherorderv}
 Let $\p$ be the solution of (\ref{phiScalarFlow}) and suppose, that there exists $0<\g\leq 1,$ such that
 \eq D^{k}\p=\O^{-\g}\ \ \forall 1\leq k\leq m-1,\eeq
 then 
 \eq \label{higherorderv1}D^{k}v=\O^{-2\g}\ \ \forall 1\leq k\leq m-2,\eeq 
\eq \label{higherorderv2} v_{j_{1}\ldots j_{m-1}}=\O^{-2\g}+v^{-1}\p_{kj_{1}\ldots j_{m-1}}\p^{k},\eeq and 
\eq v_{i_{1}\ldots i_{m}}=\O^{-2\g}+\O^{-\g}\star D^{m}\p+v^{-1}\p_{ki_{1}\ldots i_{m}}\p^{k}.\eeq
 \end{lemss}
 
 \begin{proof}
  For $k=1$ we have \eq \label{higherorderv3} v_{i_{1}}=v^{-1}\p_{ai_{1}}\p^{a}=\O^{-2\g}.\eeq
 Suppose the first claim to hold for $1\leq j\leq l\leq m-3.$ Then
 \eq D^{l+1}v=\sum_{s+r=l}D^{s}(v^{-1})\star D^{r}(\p_{ai_{1}}\p^{a})=\O^{-2\g},\eeq
 since \eq D^{s}(v^{-1})=\O^{-2\g}\ \forall 1\leq s\leq l\eeq and \eq D^{r}(\p_{ai_{1}}\p^{a})=\O^{-2\g}\ \forall 0\leq r\leq l,\eeq by Lemma \ref{derivrules}. To prove (\ref{higherorderv2}) we infer from (\ref{higherorderv3})
 \eq \begin{aligned} v_{j_{1}\ldots j_{m-1}}&=(\p_{aj_{1}}\p^{a}v^{-1})_{;j_{2}\ldots j_{m-1}}\\
 									&=\p_{aj_{1}\ldots j_{m-1}}\p^{a}v^{-1}+\O^{-2\g}, \end{aligned}\eeq
 where we used (\ref{higherorderv1}) to estimate $D^{m-2}(v^{-1}).$ Finally, we deduce
 \eq \begin{aligned} v_{i_{1}\ldots i_{m}}&=(\p_{ai_{1}}\p^{a}v^{-1})_{;i_{2}\dots i_{m}}\\
 							&=\p_{ai_{1}\ldots i_{m}}\p^{a}v^{-1}+D^{m}\p\star D(\p^{a}v^{-1})+D^{m}\p\star\O^{-\g}+\O^{-2\g}\\
							&=\p_{ai_{1}\ldots i_{m}}\p^{a}v^{-1}+D^{m}\p\star\O^{-\g}+\O^{-2\g}, \end{aligned}\eeq
where we used (\ref{higherorderv1}) and (\ref{higherorderv2}). 
 \end{proof}
 
 \begin{lemss} \label{uandphi}
  Let $\p$ be the solution of (\ref{phiScalarFlow}) and suppose, that there exists $0<\g\leq 1,$ such that
 \eq D^{k}\p=\O^{-\g}\ \ \forall 1\leq k\leq m-1,\eeq
  then \eq D^{k}u=\O^{k(1-\g)}\ \ \forall 1\leq k\leq m-1. \eeq
  \end{lemss}
  
  \begin{proof}
  Note, that
  \eq D\p=\t^{-1}Du \Rightarrow Du=\t D\p,\eeq
  where $\t=\t(u).$ Thus the claim holds for $k=1$ in view of Corollary \ref{thetaGrowth}. Suppose the claim to hold for $1\leq l\leq m-2.$ Then
  \eq D^{l+1}u=D^{l}(\t D\p)=\sum_{r+s=l}D^{s}(\t)\star D^{r}(D\p)=\O^{(l+1)(1-\g)},\eeq
  since $D^{r}(D\p)=\O^{-\g}\ \forall 0\leq r\leq l$ and 
  \eq D^{s}\t=\sum_{k_{1}+\ldots +sk_{s}=s}c_{s}\t^{(\sum_{i=1}^{s}k_{i})}(u)\prod_{i=1}^{s}\left(\frac{D^{i}u}{i!}\right)^{k_{i}}=\O^{s(1-\g)+1}, \eeq
  
  since $D^{i}u=\O^{i(1-\g)}\ \forall 1\leq i\leq s$ and 
  \eq \t^{(a)}=\begin{cases} \t, & a\ \mathrm{even}\\
  					\dot{\t}, & a\ \mathrm{odd} \end{cases} =\O^{1}. \eeq
    \end{proof}

\begin{lemss} \label{higherordertheta}
 Let $\p$ be the solution of (\ref{phiScalarFlow}) and suppose, that there exists $0<\g\leq 1,$ such that
  \eq D^{k}\p=\O^{-\g}\ \ \forall 1\leq k\leq m-1,\eeq
 then 
 \eq \label{higherordertheta1} (\t^{p-1})_{i_{1}\ldots i_{m}}=(p-1)\t^{p-1}\dot{\t}\p_{i_{1}\ldots i_{m}}+\O^{p-1+m(1-\g)}\eeq and 
  \eq \label{higherordertheta2} D^{k}(\t^{p-1})=\O^{p-1+k(1-\g)}\ \ \forall 0\leq k\leq m-1.\eeq
\end{lemss}

\begin{proof}
For the real function $f(x)=x^{p-1}$ and $g=f\circ\t$ there holds
\eq g^{(s)}=\sum_{k_{1}+\ldots+sk_{s}=s}c_{s}f^{(\sum_{i=1}^{s}k_{i})}(\t)\prod_{i=1}^{s}\left(\frac{\t^{(i)}}{i!}\right)^{k_{i}}=\O^{p-1},\eeq
since \eq f^{(a)}(\t)=\prod_{i=1}^{a}(p-i)\t^{p-1-a}=\O^{p-1-a}, \eeq $a=\sum_{i=1}^{s}k_{i},$ and
\eq \prod_{i=1}^{s}\left(\frac{\t^{(i)}}{i!}\right)^{k_{i}}=\O^{a}. \eeq
Thus, (\ref{higherordertheta2}) follows from the the di Bruno formula again, (\ref{diBruno}), applied to $g\circ u,$ and by Lemma  \ref{uandphi}. Note that (\ref{higherordertheta2}) also holds for $\dot{\t}$ instead of $\t,$ because they share the same growth behavior and there holds $\ddot{\t}=\t.$\ \\

In order to prove (\ref{higherordertheta1}), observe that \eq (\t^{p-1})_{i_{1}}=(p-1)\dot{\t}\t^{p-1}\p_{i_{1}}\eeq
 and \eq \begin{aligned} (\t^{p-1})_{i_{1}\ldots i_{m}}&=(p-1)\t^{p-1}\dot{\t}\p_{i_{1}\ldots i_{m}}\\
 						&\hphantom{=}+\sum_{\substack{s+r=m-1\\ s\geq 1}}D^{s}((p-1)\dot{\t}\t^{p-1})\star D^{r}(\p_{i_{1}})\\
 							&=(p-1)\t^{p-1}\dot{\t}\p_{i_{1}\ldots i_{m}}+\O^{p-1+m(1-\g)}, \end{aligned}\eeq
where we used $D^{k}\p=\O^{-\g}$ and (\ref{higherordertheta1}) applied to $\t$ and $\dot{\t}$ as well, also using
\eq D^{s}(\dot{\t}\t)=\sum_{s_{1}+s_{2}=s}D^{s_{1}}\dot{\t}\star D^{s_{2}}\t=\O^{1+s_{1}(1-\g)}\star\O^{p-1+s_{2}(1-\g)}.\eeq
\end{proof}

\begin{lemss} \label{gtildeandhtilde}
 Let $\p$ be the solution of (\ref{phiScalarFlow}) and suppose, that there exists $0<\g\leq 1,$ such that
 \eq D^{k}\p=\O^{-\g}\ \ \forall 1\leq k\leq m-1\eeq and set \eq \~{g}_{ij}=\p_{i}\p_{j}+\sigma_{ij},\eeq
 then 
 \eq \~{g}=\O^{0},\eeq 
\eq \label{gtildeandhtilde1} D^{k}\~{g}=\O^{-2\g}\ \ \forall 1\leq k\leq m-2,\eeq 
\eq \label{gtildeandhtilde2} D^{m-1}\~{g} =\O^{-2\g}+D^{m}\p\star\O^{-\g}, \eeq 
\eq \label{gtildeandhtilde3} D^{m}\~{g}=\O^{-2\g}+D^{m}\p\star\O^{-\g}+D^{m+1}\p\star\O^{-\g},\eeq 
\eq \label{gtildeandhtilde4}D^{k}(\~{h}^{i}_{j})=\O^{1+k(1-\g)}\ \ \forall 0\leq k \leq m-3,\eeq 
\eq D^{m-2}(\~{h}^{i}_{j})=\O^{1+(m-2)(1-\g)}+D^{m}\p\star\O^{0},\eeq 
\eq D^{m-1}(\~{h}^{i}_{j})=\O^{1+(m-1)(1-\g)}+D^{m}\p\star\O^{1-\g}+D^{m+1}\p\star\O^{0},\eeq and 
\eq\begin{aligned} \label{gtildeandhtilde5} \~{h}^{l}_{a;i_{1}\ldots i_{m}}&=-v^{-1}\~{g}^{lr}\p_{ra;i_{1}\ldots i_{m}}+v^{-1}\t^{2}\p_{i_{1}\ldots i_{m}}\d^{l}_{a} \\
 									&\hphantom{=}+\O^{1+m(1-\g)}+D^{m}\p\star\O^{1-\g}+D^{m+1}\p\star\O^{1-\g}\\
									&\hphantom{=}+D^{m}\p\star D^{m}\p\star\O^{-\g}. \end{aligned}\eeq 
 \end{lemss}
 \vspace{0,01 cm}

\begin{proof}
We have \eq \~{g}^{lr}=\sigma^{lr}-v^{-2}\p^{l}\p^{r}=\O^{0}.\eeq
For all $1\leq k\leq m$ we deduce
\eq \begin{aligned} D^{k}(\~{g}^{lr})&=-\Big(D^{k}(v^{-2})\p^{l}\p^{r}+v^{-2}D^{k}(\p^{l}\p^{r}) \\
							&\hphantom{=}+\sum_{\substack{s+t=k \\ s,t\geq 1}}D^{s}(v^{-2})\star D^{t}(\p^{l}\p^{r})\Big), \end{aligned} \eeq
from which (\ref{gtildeandhtilde1})-(\ref{gtildeandhtilde3}) follow by Lemma \ref{higherorderv}.
There holds \eq \~{h}^{l}_{a}=-v^{-1}(\~{g}^{lr}\p_{ra}-\dot{\t}\d^{l}_{a}).\eeq
For all $1\leq k\leq m$ we have
\eq \begin{aligned} \label{gtildeandhtilde6}
D^{k}(\~{h}^{l}_{a})&=-(D^{k}(v^{-1})\~{g}^{lr}\p_{ra}+v^{-1}D^{k}(\~{g}^{lr})\p_{ra}+v^{-1}\~{g}^{lr}D^{k}(\p_{ra}))\\
			&\hphantom{=}+D^{k}(v^{-1})\dot{\t}\d^{l}_{a}+v^{-1}D^{k}(\dot{\t})\d^{l}_{a}\\
			&\hphantom{=}+\sum_{\substack{j_{1}+j_{2}+j_{3}=k \\ \exists i_{1}\neq i_{2}\colon j_{i_{1}},j_{i_{2}}\geq 1}}D^{j_{1}}(v^{-1})\star D^{j_{2}}(\~{g}^{lr})\star D^{j_{3}}(\p_{ra})\\
			&\hphantom{=}+\sum_{\substack{j_{1}+j_{2}=k \\ j_{1},j_{2}\geq 1}}D^{j_{1}}(v^{-1})\star D^{j_{2}}(\dot{\t})\d^{l}_{a}. \end{aligned}\eeq
In order to prove (\ref{gtildeandhtilde4})-(\ref{gtildeandhtilde5}), we examine (\ref{gtildeandhtilde6}) and use Lemma \ref{higherorderv}, Lemma \ref{higherordertheta} and (\ref{gtildeandhtilde1})-(\ref{gtildeandhtilde3}). First let $1\leq k\leq m-3.$  Then
\eq D^{k}\~{h}^{l}_{a}=\O^{-3\g}+\O^{-\g}+\O^{1-2\g}+\O^{1+k(1-\g)}=\O^{1+k(1-\g)}, \eeq
\eq \begin{aligned} D^{m-2}\~{h}^{l}_{a}&=\O^{-3\g}+D^{m}\p\star\O^{0}+\O^{1-2\g}+\O^{1+(m-2)(1-\g)}\\
						&=\O^{1+(m-2)(1-\g)}+\O^{0}\star D^{m}\p, \end{aligned}\eeq
\eq \begin{aligned} D^{m-1}\~{h}^{l}_{a}&=\O^{-3\g}+D^{m}\p\star\O^{-2\g}+D^{m+1}\p\star\O^{0}+\O^{1-2\g}\\
							&\hphantom{=}+D^{m}\p\star\O^{1-\g}+\O^{1+(m-1)(1-\g)}\\
							&=D^{m}\p\star\O^{1-\g}+D^{m+1}\phi\star\O^{0}+\O^{1+(m-1)(1-\g)} \end{aligned}\eeq
and finally
\eq \begin{aligned} \~{h}^{l}_{a;i_{1}\ldots i_{m}}&=\O^{-3\g}+D^{m}\p\star\O^{-2\g}+D^{m+1}\p\star\O^{-2\g}\\
								&\hphantom{=}-v^{-1}\~{g}^{lr}\p_{ra;i_{1}\ldots i_{m}}+\O^{1-2\g}+D^{m}\p\star\O^{1-\g}\\
								&\hphantom{=}+D^{m+1}\p\star\O^{1-\g}+\O^{1+m(1-\g)} +v^{-1}\t^{2}\p_{i_{1}\ldots i_{m}}\d^{l}_{a}\\
								&\hphantom{=}+D^{m}\p\star D^{m}\p\star \O^{-\g} \\
								&=\O^{1+m(1-\g)}+D^{m}\p\star\O^{1-\g}+D^{m+1}\p\star\O^{1-\g}\\
								&\hphantom{=}+D^{m}\p\star D^{m}\p\star\O^{-\g}- v^{-1}\~{g}^{lr}\p_{ra;i_{1}\ldots i_{m}}\\					&\hphantom{=}+v^{-1}\t^{2}\p_{i_{1}\ldots i_{m}}\d^{l}_{a}. \end{aligned} \eeq
\end{proof}

\begin{lemss} \label{PhicircF}
Let $\p$ be the solution of (\ref{phiScalarFlow}) and suppose, that there exists $0<\g\leq 1,$ such that
 \eq D^{k}\p=\O^{-\g}\ \ \forall 1\leq k\leq m-1,\eeq
  then \eq \mathcal{D}^{\a}\Phi=\O^{-(p+\a)}\ \ \forall 0\leq\a\leq m, \eeq
 where $\mathcal{D}^{\a}$ denotes an arbitrary derivative of order $\a$ with respect to the argument $\~{h}^{l}_{a}.$
\end{lemss}

\begin{proof}
Di Bruno's formula, (\ref{diBruno}), gives
\eq \mathcal{D}^{\a}\Phi=\sum_{k_{1}+\ldots+\a k_{\a}=\a}c_{\a}\Phi^{(\sum_{i=1}^{\a}k_{i})}(F)\prod_{i=1}^{\a}\left(\frac{\mathcal{D}^{i}F}{i!}\right)^{k_{i}},\eeq
where $\Phi^{(r)}=\frac{d^{r}}{ds^{r}}\Phi(s).$
In view of \eq F=F(\~{h}^{l}_{a})=F(\t h^{l}_{a}), \eeq
we have \eq F=\O^{1}\eeq
and by homogeneity \eq \mathcal{D}^{i}F=\O^{1-i},\eeq
as well as \eq \Phi^{(r)}(F)=\O^{-(p+r)}.\eeq
Thus \eq \prod_{i=1}^{\a}\left(\frac{\mathcal{D}^{i}F}{i!}\right)^{k_{i}}=\prod_{i=1}^{\a}\frac{\O^{k_{i}-ik_{i}}}{i!}=\O^{\sum_{i=1}^{\a}k_{i}-\sum_{i=1}^{\a}ik_{i}},\eeq
which implies 
\eq \mathcal{D}^{\a}\Phi=\sum_{k_{1}+\ldots+\a k_{\a}=\a}\O^{-(p+\sum_{i=1}^{\a}k_{i})}\star\O^{\sum_{i=1}^{\a}k_{i}-\sum_{i=1}^{\a}ik_{i}}=\O^{-(p+\a)}.\eeq
\end{proof}

\begin{lemss} \label{higherorderPhi}
Let $\p$ be the solution of (\ref{phiScalarFlow}) and suppose, that there exists $0<\g\leq 1,$ such that
 \eq D^{k}\p=\O^{-\g}\ \ \forall 1\leq k\leq m-1,\eeq
then 
\eq D^{k}\Phi=\O^{k(1-\g)-p}\ \ \forall 0\leq k\leq m-3,\eeq
\eq \ D^{m-2}\Phi=\O^{(m-2)(1-\g)-p}+D^{m}\p\star \O^{-(p+1)},\eeq
\eq \begin{aligned} D^{m-1}\Phi&=\O^{(m-1)(1-\g)-p}+D^{m}\p\star\O^{-(p+\g)}\\
				&\hphantom{=}+D^{m}\p\star D^{m}\p\star\O^{-(p+2)}+D^{m+1}\p\star\O^{-(p+1)} \end{aligned} \eeq and
\eq  \begin{aligned} \label{higherorderPhi1}
 \Phi_{i_{1}\ldots i_{m}}&=-\dot{\Phi}v^{-1}F^{a}_{l}\~{g}^{lr}\p_{ra;i_{1}\ldots i_{m}}+\dot{\Phi}v^{-1}F^{a}_{l}\t^{2}\p_{i_{1}\ldots i_{m}}\d^{l}_{a}\\
				&\hphantom{=}+\O^{m(1-\g)-p}+D^{m}\p\star\O^{1-2\g-p}\\
				&\hphantom{=}+D^{m}\p\star D^{m}\p\star\O^{-(p+1+\g)}+D^{m+1}\p\star\O^{-(p+\g)}\\
				&\hphantom{=}+D^{m}\p\star D^{m+1}\p\star\O^{-(p+2)}\\
				&\hphantom{=}+D^{m}\p\star D^{m}\p\star D^{m}\p\star\O^{-(p+3)}. \end{aligned}\eeq 									
\end{lemss}

\begin{proof}
We consider $\Phi(\~{h}^{i}_{j})\equiv\Phi(F(\~{h}^{i}_{j})).$
\eq D^{\b}\Phi=\sum_{k_{1}+\ldots+\b k_{\b}=\b}c_{\b}\mathcal{D}^{\sum_{i=1}^{\b}k_{i}}\Phi(\~{h}^{l}_{a})\prod_{i=1}^{\b}\left(\frac{D^{i}\~{h}^{l}_{a}}{i!}\right)^{k_{i}}. \eeq
We consider the different cases separately and use Lemma \ref{gtildeandhtilde} and Lemma \ref{PhicircF} to obtain, that
if $\b\leq m-3,$ then
\eq D^{\b}\Phi=\O^{-(p+\sum_{i=1}^{\b}k_{i})}\star\O^{\sum_{i=1}^{\b}k_{i}+\sum_{i=1}^{\b}ik_{i}(1-\g)}=\O^{\b(1-\g)-p}.\eeq
If $\b=m-2,$ then \eq\begin{aligned}  D^{\b}\Phi&=\sum_{k_{1}+\ldots \b k_{\b}=\b}\O^{-(p+\sum_{i=1}^{\b}k_{i})}\star\prod_{i=1}^{\b-1}\left(\frac{D^{i}\~{h}^{l}_{a}}{i!}\right)^{k_{i}}\\
&\hphantom{=}\star (\O^{1+(m-2)(1-\g)}+D^{m}\p\star\O^{0})^{k_{m-2}}\\
&=\O^{\b(1-\g)-p}+D^{m}\p\star\O^{-(p+1)}.   \end{aligned} \eeq
For $\b=m-1$ we get
\eq \begin{aligned} D^{\b}\Phi&=\sum_{k_{1}+\ldots+\b k_{\b}=\b}\O^{-(p+\sum_{i=1}^{\b}k_{i})}\star\prod_{i=1}^{\b-2}\left(\frac{D^{i}\~{h}^{l}_{a}}{i!}\right)^{k_{i}}\\
		&\hphantom{=}\star(\O^{1+(m-2)(1-\g)}+D^{m}\p\star\O^{0})^{k_{m-2}}\\
		&\hphantom{=}\star (\O^{1+(m-1)(1-\g)}+D^{m}\p\star\O^{1-\g}+D^{m+1}\p\star\O^{0})^{k_{m-1}}\\
		&=\O^{\b(1-\g)-p}+D^{m}\p\star\O^{-(p+\g)}+D^{m}\p\star D^{m}\p\star\O^{-(p+2)}\\
		&\hphantom{=}+D^{m+1}\p\star \O^{-(p+1)}. \end{aligned} \eeq
In order to prove (\ref{higherorderPhi1}), we calculate \eq \begin{aligned}
\Phi_{i_{1}\ldots i_{m}}&=\sum_{k_{1}+\ldots mk_{m}=m}\frac{m!}{k_{1}!\cdots k_{m}!}\Phi^{(\sum_{i=1}^{m}k_{i})}\prod_{i=1}^{m-3}\left(\frac{D^{i}\~{h}^{l}_{a}}{i!}\right)^{k_{i}}\\
		&\hphantom{=}\star(\O^{1+(m-2)(1-\g)}+D^{m}\p\star\O^{0})^{k_{m-2}}\\
		&\hphantom{=}\star(\O^{1+(m-1)(1-\g)}+D^{m}\p\star\O^{1-\g}+D^{m+1}\p\star\O^{0})^{k_{m-1}}\\
		&\hphantom{=}\star\Big(-\frac{1}{m!}v^{-1}\~{g}^{lr}\p_{ra;i_{1}\ldots i_{m}}+v^{-1}\t^{2}\p_{i_{1}\ldots i_{m}}\d^{l}_{a}\\
		&\hphantom{=}+\O^{1+m(1-\g)}+D^{m}\p\star\O^{1-\g}+D^{m+1}\p\star\O^{1-\g}\\
		&\hphantom{=}+D^{m}\p\star D^{m}\p\star\O^{-\g}\Big)^{k_{m}}.\end{aligned}\eeq
Thus \eq\begin{aligned}
		\Phi_{i_{1}\ldots i_{m}}&=-\dot{\Phi}F^{a}_{l}v^{-1}\~{g}^{lr}\p_{ra;i_{1}\ldots i_{m}}+\dot{\Phi}F^{a}_{l}v^{-1}\t^{2}\p_{i_{1}\ldots i_{m}}\d^{l}_{a}\\
		&\hphantom{=}+\O^{m(1-\g)-p}+D^{m}\p\star\O^{-(p+\g)}+D^{m+1}\p\star\O^{-(p+\g)}\\
		&\hphantom{=}+D^{m}\p\star D^{m}\p\star\O^{-(p+1+\g)}+\O^{m(1-\g)-p}\\
		&\hphantom{=}+D^{m}\p\star\O^{1-2\g-p}+D^{m+1}\p\star\O^{-(p+\g)}\\
		&\hphantom{=}+D^{m}\p\star D^{m}\p\star\O^{-(p+1+\g)}+D^{m}\p\star D^{m+1}\p\star\O^{-(p+2)}\\
		&\hphantom{=}+\O^{m(1-\g)-p}+D^{m}\p\star\O^{1-2\g-p}+D^{m}\p\star D^{m}\p\star\O^{-(p+2)}\\
		&\hphantom{=}+\O^{m(1-\g)-p}+D^{m}\p\star\O^{1-2\g-p}+\O^{m(1-\g)-p}\\
		&\hphantom{=}+D^{m}\p\star\O^{1-2\g-p}+D^{m}\p\star D^{m}\p\star\O^{-(p+1+\g)}\\
		&\hphantom{=}+D^{m}\p\star D^{m}\p\star D^{m}\p\star \O^{-(p+3)},\end{aligned} \eeq
so that finally\eq \begin{aligned}
		\Phi_{i_{1}\ldots i_{m}}&=-\dot{\Phi}F^{a}_{l}v^{-1}\~{g}^{lr}\p_{ra;i_{1}\ldots i_{m}}+\dot{\Phi}F^{a}_{l}v^{-1}\t^{2}\p_{i_{1}\ldots i_{m}}\d^{l}_{a}\\
		&\hphantom{=}+\O^{m(1-\g)-p}+D^{m}\p\star\O^{1-2\g-p}\\
		&\hphantom{=}+D^{m}\p\star D^{m}\p\star\O^{-(p+1+\g)}+D^{m+1}\p\star\O^{-(p+\g)}\\
		&\hphantom{=}+D^{m}\p\star D^{m+1}\p\star\O^{-(p+2)}\\
		&\hphantom{=}+D^{m}\p\star D^{m}\p\star D^{m}\p\star\O^{-(p+3)}. \end{aligned} \eeq
\end{proof}

\begin{lemss} \label{RicciId}
For a function \eq \p\colon \S^{n}\rightarrow\R \eeq
there holds \eq \p_{rk;i_{1}\dots i_{m}}=\p_{i_{1}\ldots i_{m};rk}+D^{m}\p\star\O^{0}.\eeq
\end{lemss}

\begin{proof}
We shift $i_{j}$ into the $j$-th position inductively. For $j=1$ we have
\eq \begin{aligned} \p_{rk;i_{1}\ldots i_{m}}&=\p_{ri_{1};ki_{2}\ldots i_{m}}+({R^{s}}_{rki_{1}}\p_{s})_{;i_{2}\ldots i_{m}}\\
								&=\p_{ri_{1}ki_{2}\ldots i_{m}}+(\d^{s}_{k}\sigma_{ri_{1}}-\d^{s}_{i_{1}}\sigma_{rk})\p_{si_{2}\ldots i_{m}}\\
								&=\p_{i_{1}rki_{2}\ldots i_{m}}+D^{m}\p\star \O^{0}. \end{aligned} \eeq
	Suppose inductively
\eq \p_{rk;i_{1}\ldots i_{m}}=\p_{i_{1}\ldots i_{j}rki_{j+1}\ldots i_{m}}+D^{m}\p\star\O^{0}, \eeq
then
\eq \begin{aligned} \p_{rki_{1}\ldots i_{m}}&=\p_{i_{1}\ldots i_{j}ri_{j+1}ki_{j+2}\ldots i_{m}}\\
							&\hphantom{=}+\left(\sum_{l=1}^{j}{R^{s_{l}}}_{i_{l}ki_{j+1}}\p_{i_{1}\ldots i_{l-1}s_{l}i_{l+1}\ldots i_{j}r}\right)_{;i_{j+2}\ldots i_{m}}\\
							&\hphantom{=}+({R^{s}}_{rki_{j+1}}\p_{i_{1}\ldots i_{j}s})_{;i_{j+2}\ldots i_{m}}+D^{m}\p\star\O^{0}\\
							&=\p_{i_{1}\ldots i_{j}ri_{j+1}ki_{j+2}\ldots i_{m}}+D^{m}\p\star\O^{0} \end{aligned} \eeq
	and analogously for exchanging $r$ and $i_{j+1}.$								
\end{proof}

\begin{lemss} \label{scalarflowderivatives}
Let $\p$ be the solution of (\ref{phiScalarFlow}) and suppose, that there exists $0<\g\leq 1,$ such that
 \eq D^{k}\p=\O^{-\g}\ \ \forall 1\leq k\leq m-1,\eeq then the functions \eq z=\frac 12 \abs{D^{m-1}\p}^{2}=\frac 12\p_{i_{1}\ldots i_{m-1}}\p^{i_{1}\ldots i_{m-1}}\eeq
 and \eq w=\frac 12\abs{D^{m}\p}^{2}=\frac 12 \p_{i_{1}\ldots i_{m}}\p^{i_{1}\ldots i_{m}}\eeq
 satisfy
 \eq \begin{aligned} \label{scalarflowderivatives1} \dot{z}-\t^{p-1}\dot{\Phi}F^{ar}z_{ar}&=-\t^{p-1}\dot{\Phi}F^{ar}\p_{i_{1}\ldots i_{m-1};r}{\p^{i_{1}\ldots i_{m-1}}}_{;a}\\
 									&\hphantom{=}+\O^{-(1+\g)+(m-1)(1-\g)} \\
 								&\hphantom{=}+D^{m}\p\star\O^{-(1+2\g)}+D^{m}\p\star D^{m}\p\star\O^{-(3+\g)} \end{aligned} \eeq and					
\eq \begin{aligned} \label{scalarflowderivatives2} \dot{w}-\t^{p-1}\dot{\Phi}F^{ar}w_{ar}&=-2(p-1)\t^{p-1}\dot{\t}v\Phi w-2\t^{p+1}\dot{\Phi}F^{a}_{a}w\\
				&\hphantom{=}-\t^{p-1}\dot{\Phi}F^{ar}\p_{i_{1}\ldots i_{m};r}{\p^{i_{1}\ldots i_{m}}}_{;a}\\
				&\hphantom{=}+D^{m}\p\star\O^{-1+m(1-\g)}+D^{m}\p\star D^{m}\p\star\O^{-2\g}\\
				&\hphantom{=}+D^{m}\p\star D^{m}\p\star D^{m}\p\star\O^{-(2+\g)}\\
				&\hphantom{=}+D^{m}\p\star D^{m}\p\star D^{m}\p\star D^{m}\p\star\O^{-4}\\
				&\hphantom{=}+D^{m}\p\star D^{m+1}\p\star\O^{-(1+\g)}\\
				&\hphantom{=}+D^{m}\p\star D^{m}\p\star D^{m+1}\p\star\O^{-3}. \end{aligned}\eeq				
\end{lemss}

\begin{proof}
$\p$ satisfies \eq \dot{\p}=-\t^{p-1}v\Phi\ \mathrm{on}\ [0,\infty)\times \S^{n}. \eeq
Differentiating covariantly with respect to $\sigma_{ij}$ gives
\eq \begin{aligned} \dot{\p}_{i_{1}\ldots i_{k}}&=-(\t^{p-1})_{i_{1}\ldots i_{k}}v\Phi-\t^{p-1}v_{i_{1}\ldots i_{k}}\Phi-\t^{p-1}v\Phi_{i_{1}\ldots i_{k}}\\
									&\hphantom{=}+\sum_{\substack{j_{1}+j_{2}+j_{3}=k\\ \exists s\neq t\colon j_{s},j_{t}\neq 0}}D^{j_{1}}(\t^{p-1})\star D^{j_{2}}v\star D^{j_{3}}\Phi. \end{aligned}\eeq
In order to prove (\ref{scalarflowderivatives1}), we consider $k=m-1$ and obtain	
\eq \begin{aligned} \dot{\p}_{i_{1}\ldots i_{m-1}}&=\O^{-1+(m-1)(1-\g)}+D^{m}\p\star\O^{-(1+\g)}\\
				&\hphantom{=}+D^{m}\p\star D^{m}\p\star\O^{-3}+\t^{p-1}\dot{\Phi}F^{a}_{l}\~{g}^{lr}\p_{ra;i_{1}\ldots i_{m-1}}. \end{aligned}\eeq	
There holds \eq z_{ra}=\p_{i_{1}\ldots i_{m-1}ra}\p^{i_{1}\ldots i_{m-1}}+\p_{i_{1}\ldots i_{m-1}r}{\p^{i_{1}\ldots i_{m-1}}}_{;a} \eeq
and thus
\eq \begin{aligned} \p_{ra;i_{1}\ldots i_{m-1}}\p^{i_{1}\ldots i_{m-1}}&=\p_{i_{1}\ldots i_{m-1};ra}\p^{i_{1}\ldots i_{m-1}}\\
									&\hphantom{=}+D^{m-1}\p\star D^{m-1}\p\star \O^{0}\\
											&=z_{ra}-\p_{i_{1}\ldots i_{m-1};r}{\p^{i_{1}\ldots i_{m-1}}}_{;a}\\
											&\hphantom{=}+D^{m-1}\p\star D^{m-1}\p\star\O^{0}.
											\end{aligned}\eeq	
We conclude, that
\eq \begin{aligned} \dot{z}-\t^{p-1}\dot{\Phi}F^{a}_{l}\~{g}^{lr}z_{ra}&=-\t^{p-1}\dot{\Phi}F^{a}_{l}\~{g}^{lr}\p_{i_{1}\ldots i_{m-1};r}{\p^{i_{1}\ldots i_{m-1}}}_{;a}\\
											&\hphantom{=}+\O^{-(1+\g)+(m-1)(1-\g)}+D^{m}\p\star\O^{-(1+2\g)}\\
											&\hphantom{=}+D^{m}\p\star D^{m}\p\star\O^{-(3+\g)}.
											\end{aligned}\eeq
To prove (\ref{scalarflowderivatives2}), set $k=m$ to obtain
\eq \begin{aligned}
\dot{\p}_{i_{1}\ldots i_{m}}&=-(p-1)\t^{p-1}\dot{\t}\p_{i_{1}\ldots i_{m}}v\Phi-\t^{p-1}v^{-1}\p_{ai_{1}\ldots i_{m}}\p^{a}\Phi\\
					&\hphantom{=}+\O^{-1+m(1-\g)}+D^{m}\p\star\O^{-2\g}+\t^{p-1}\dot{\Phi}F^{ar}\p_{ra;i_{1}\dots i_{m}}\\
					&\hphantom{=}-\t^{p+1}\dot{\Phi}F^{a}_{a}\p_{i_{1}\ldots i_{m}}	+D^{m}\p\star D^{m}\p\star\O^{-(2+\g)}\\
					&\hphantom{=}+D^{m+1}\p\star \O^{-(1+\g)}+D^{m}\p\star D^{m+1}\p\star\O^{-3}\\
					&\hphantom{=}+D^{m}\p\star D^{m}\p\star D^{m}\p\star\O^{-4}. \end{aligned}\eeq 
As above we have
\eq w_{ra}=\p_{i_{1}\ldots i_{m};ra}\p^{i_{1}\ldots i_{m}}+\p_{i_{1}\ldots i_{m};r}{\p^{i_{1}\ldots i_{m}}}_{;a}, \eeq 
\eq \begin{aligned} \p_{ra;i_{1}\ldots i_{m}}\p^{i_{1}\ldots i_{m}}&=\p_{i_{1}\ldots i_{m};ra}\p^{i_{1}\ldots i_{m}}+D^{m}\p\star D^{m}\p\star \O^{0}\\
									&=w_{ra}-\p_{i_{1}\ldots i_{m};r}{\p^{i_{1}\ldots i_{m}}}_{;a}+D^{m}\p\star D^{m}\p\star \O^{0} \end{aligned}\eeq
and thus
\eq \begin{aligned}
\dot{w}-\t^{p-1}\dot{\Phi}F^{ar}w_{ar}&=-2(p-1)\t^{p-1}\dot{\t}v\Phi w+D^{m}\p\star \O^{-1+m(1-\g)}\\
						&\hphantom{=}+D^{m}\p\star D^{m+1}\p\star \O^{-(1+\g)}\\
						&\hphantom{=}+D^{m}\p\star D^{m}\p\star \O^{-2\g}\\
						&\hphantom{=}-\t^{p-1}\dot{\Phi}F^{ar}\p_{i_{1}\ldots i_{m};r}{\p^{i_{1}\ldots i_{m}}}_{;a}-2\t^{p+1}\dot{\Phi}F^{a}_{a} w\\
						&\hphantom{=}+D^{m}\p\star D^{m}\p\star D^{m}\p\star \O^{-(2+\g)}\\
						&\hphantom{=}+D^{m}\p\star D^{m}\p\star D^{m+1}\p\star \O^{-3}\\
						&\hphantom{=}+D^{m}\p\star D^{m}\p\star D^{m}\p\star D^{m}\p\star \O^{-4}. \end{aligned}\eeq
\end{proof}

\begin{thmss} \label{phiderivativedecay}
Let $\p$ be the solution of (\ref{phiScalarFlow}), then
\eq D^{m}\p=\O^{-\g}\ \ \forall m\in \mathbb{N}^{*}\ \ \forall 0\leq\g<1. \eeq
\end{thmss}

\begin{proof}
We use a method similar to the proof of \cite[Lemma 6.10]{cg:ImcfHyp}.\\

For $m=1,2$ this has been proven for $\g=1,$ cf. Theorem \ref{optGradDecay} and Theorem \ref{phi2DerDecay}. Thus let the conclusion hold for $1\leq k\leq m-1,$ $m\geq 3.$ Let 
\eq z=\frac 12 \abs{D^{m-1}\p}^{2}\eeq and \eq w=\frac 12\abs{D^{m}\p}^{2}, \eeq as well as
\eq \~{w}=we^{\frac{2\lambda}{n^{p}}t},\ 0\leq \lambda \leq 1.\eeq
Set \eq \zeta=\log\~{w} +z. \eeq
Then by \ref{scalarflowderivatives} we have
\eq \begin{aligned}
\dot{\zeta}-\t^{p-1}\dot{\Phi}F^{a}_{l}\~{g}^{lr}\zeta_{ra}&=\~{w}^{-1}(\dot{\~{w}}-\t^{p-1}\dot{\Phi}F^{a}_{l}\~{g}^{lr}\~{w}_{ar})
				\\
						&\hphantom{=}+\t^{p-1}\dot{\Phi}F^{a}_{l}\~{g}^{lr}(\log \~{w})_{a}(\log\~{w})_{r}\\
						&\hphantom{=}+\dot{z}-\t^{p-1}\dot{\Phi}F^{a}_{l}\~{g}^{lr}z_{ra}\\
						&=w^{-1}(\dot{w}-\t^{p-1}\dot{\Phi}F^{a}_{l}\~{g}^{lr}w_{ra})+\frac{2}{n^{p}}\lambda \\
						&\hphantom{=}+\t^{p-1}\dot{\Phi}F^{a}_{l}\~{g}^{lr}(\log w)_{a}(\log w)_{r}\\
						&\hphantom{=}+\dot{z}-\t^{p-1}\dot{\Phi}F^{a}_{l}\~{g}^{lr}z_{ar}\\
						&=-2(p-1)\t^{p-1}\dot{\t}v\Phi-2\t^{p+1}\dot{\Phi}F^{a}_{a}\\
						&\hphantom{=}-\t^{p-1}\dot{\Phi}\abs{D^{m+1}\p}^{2}w^{-1}\\
						&\hphantom{=}+\t^{p-1}\dot{\Phi}(\sigma^{ar}-F^{ar})\p_{i_{1}\ldots i_{m};r}{\p^{i_{1}\ldots i_{m}}}_{;a}w^{-1}\\
						&\hphantom{=}+(\O^{-1+m(1-\g)}\star D^{m}\p)w^{-1}\\
						&\hphantom{=}+(\O^{-2\g}\star D^{m}\p\star D^{m}\p)w^{-1}\\
						&\hphantom{=}+(\O^{-(2+\g)}\star D^{m}\p\star D^{m}\p\star D^{m}\p)w^{-1}\\
						&\hphantom{=}+(\O^{-4}\star D^{m}\p\star D^{m}\p\star D^{m}\p\star D^{m}\p)w^{-1}\\
						&\hphantom{=}+(\O^{-(1+\g)}\star D^{m}\p\star D^{m+1}\p)w^{-1}\\
						&\hphantom{=}+(\O^{-3}\star D^{m}\p\star D^{m}\p\star D^{m+1}\p)w^{-1}\\
						&\hphantom{=}-\t^{p-1}\dot{\Phi}\abs{D^{m}\p}^{2}\\
						&\hphantom{=}+\t^{p-1}\dot{\Phi}(\sigma^{ar}-F^{ar})\p_{i_{1}\ldots i_{m-1};r}{\p^{i_{1}\ldots i_{m-1}}}_{;a}\\
						&\hphantom{=}+\O^{-(1+\g)+(m-1)(1-\g)}+\O^{-(1+2\g)}\star D^{m}\p\\
						&\hphantom{=}+\O^{-(3+\g)}\star D^{m}\p\star D^{m}\p\\
						&\hphantom{=}+\t^{p-1}\dot{\Phi}F^{ar}(\log w)_{a}(\log w)_{r}+\frac{2}{n^{p}}\lambda. \end{aligned}\eeq
We want to bound $\zeta.$ Thus, fix $0<T<\infty$ and suppose that \eq \sup_{[0,T]\times \S^{n}}\zeta=\zeta(t_{0},x_{0}),\ t_{0}>0.\eeq
At this point we have
\eq -z_{a}=(\log w)_{a} \eeq
and thus
\eq \begin{aligned}
 &\t^{p-1}\dot{\Phi}F^{a}_{l}\~{g}^{lr}(\log w)_{a}(\log w)_{r}\\
               =&\t^{p-1}\dot{\Phi}F^{a}_{l}\~{g}^{lr}\p_{i_{1}\ldots i_{m-1};a}\p^{i_{1}\ldots i_{m-1}}\p_{j_{1}\ldots j_{m-1};r}\p^{j_{1}\ldots j_{m-1}}\\
 				=&\O^{-2(1+\g)}\star D^{m}\p\star D^{m}\p. \end{aligned}\eeq
Thus, at $(t_{0},x_{0})$, also supposing that \eq\abs{D^{m}\p}e^{\frac{\lambda}{n^{p}}t}\geq 1,\eeq
\eq \begin{aligned}
0&\leq 2\t^{p-1}\dot{\t}v\Phi-2(\t^{p+1}\dot{\Phi}F^{a}_{a}+p\t^{p-1}\dot{\t}v\Phi)\\
  &\hphantom{=}+w^{-1}(-\t^{p-1}\dot{\Phi}\abs{D^{m+1}\p}^{2}\\
  &\hphantom{=}+\t^{p-1}\dot{\Phi}(\sigma^{ar}-F^{ar})\p_{i_{1}\ldots i_{m};r}{\p^{i_{1}\ldots i_{m}}}_{;a}\\
  &\hphantom{=}+c_{m}e^{-\frac{1+\g}{n^{p}}t}\abs{D^{m}\p}\abs{D^{m+1}\p}+c_{m}e^{-\frac{3}{n^{p}}t}\abs{D^{m}\p}^{2}\abs{D^{m+1}\p})\\
  &\hphantom{=}+(-\t^{p-1}\dot{\Phi}\abs{D^{m}\p}^{2}+\t^{p-1}\dot{\Phi}(\sigma^{ar}-F^{ar})\p_{i_{1}\ldots i_{m-1};r}{\p^{i_{1}\ldots i_{m-1}}}_{;a}\\
  &\hphantom{=}+c_{m}e^{\frac{m(1-\g)-1}{n^{p}}t}\abs{D^{m}\p}^{-1}+c_{m}e^{\frac{(m-1)(1-\g)-(1+\g)}{n^{p}}t}\\
  &\hphantom{=}+c_{m}e^{-\frac{1+2\g}{n^{p}}t}\abs{D^{m}\p}+c_{m}e^{-\frac{2(1+\g)}{n^{p}}t}\abs{D^{m}\p}^{2}) +\frac{2}{n^{p}}\lambda\\
  &\leq -2\t^{-1}\dot{\t}vF^{-p}(h^{l}_{a})+2pF^{-p}(h^{l}_{a})(\t^{-1}\dot{\t}v-F^{-1}F^{a}_{a})\\
  &\hphantom{=}+w^{-1}(-\t^{p-1}\dot{\Phi}\abs{D^{m+1}\p}^{2}\\
  &\hphantom{=}+\t^{p-1}\dot{\Phi}(\sigma^{ar}-F^{ar})\p_{i_{1}\ldots i_{m};r}{\p^{i_{1}\ldots i_{m}}}_{;a}\\
  &\hphantom{=}+ce^{-\frac{\g}{n^{p}}t}\abs{D^{m}\p}^{2}+ce^{-\frac{2+\g}{n^{p}}t}\abs{D^{m+1}\p}^{2}\\
  &\hphantom{=}+ce^{-\frac{3}{n^{p}}t}(\abs{D^{m}\p}^{4}+\abs{D^{m+1}\p}^{2}))\\
  &\hphantom{=}+(-\t^{p-1}\dot{\Phi}\abs{D^{m}\p}^{2}+\t^{p-1}\dot{\Phi}(\sigma^{ar}-F^{ar})\p_{i_{1}\ldots i_{m-1};r}{\p^{i_{1}\ldots i_{m-1}}}_{;a}\\
  &\hphantom{=}+ce^{\frac{\lambda +m(1-\g)-1}{n^{p}}t}+ce^{\frac{(m-1)(1-\g)-(1+\g)}{n^{p}}t}+ce^{-\frac{1+2\g}{n^{p}}t}\abs{D^{m}\p}^{2})\\
  &\hphantom{=}+\frac{2}{n^{p}}\lambda, \end{aligned}\eeq
  
  where we used \eq ab\leq \frac{\e}{2}a^{2}+\frac{1}{2\e}b^{2}\eeq with $a=\abs{D^{m}\p},$ $b=\abs{D^{m+1}\p}$ and $\e=e^{\frac{t}{n^{p}}},$ as well as with
  $a=1,$ $b=\abs{D^{m}\p}$ and $\e=1.$
  
  From the $C^{1}$ and $C^{2}$ estimates we know \eq -2\t^{-1}\dot{\t}vF^{-p}\rightarrow -\frac{2}{n^{p}},\eeq
  \eq \limsup\limits_{t\rightarrow\infty}2pF^{-p}(\t^{-1}\dot{\t}v-F^{-1}F^{a}_{a})\leq 0\eeq and
  \eq \abs{\sigma^{ar}-F^{ar}}\rightarrow 0.\eeq
  In view of \eq \t^{p-1}\dot{\Phi}=p\t^{-2}F^{-(p+1)}(h^{l}_{a})\geq ce^{-\frac{2}{n^{p}}t}, \eeq cf. Corollary \ref{thetaGrowth},
  we may absorb any bad term by the good terms
  \eq -\t^{p-1}\dot{\Phi}\abs{D^{k}\p}^{2},\ k=m, m+1, \eeq if $t_{0}$ is supposed to be large enough and $0\leq\lambda <1.$   Thus, for large $t_{0}$ and $\lambda <1$ we obtain a contradiction and conclude
  \eq\abs{D^{m}\p}e^{\frac{\lambda}{n^{p}}t}\leq c=c(n,p,M_{0},m,\lambda)\ \  \forall 0\leq\lambda<1,\eeq
  which means \eq D^{m}\p=\O^{-\g}\ \ \forall 0\leq\g<1.\eeq

\end{proof}

\section{The conformally flat parametrization and convergence to a sphere}

In order to complete the proof of Theorem \ref{mainthm}, we use the conformally flat parametrization and consider the flow in $\R^{n+1}.$ From now on, we distinguish quantities in $\H^{n+1}$ from those in $\R^{n+1}$ by an additional br\`eve, e.g. $\breve{u}, \breve{g}_{ij},$ etc., compare \cite[ch. 5]{cg:ImcfHyp}. For a flow hypersurface
\eq M=\graph \breve{u}=\graph u\eeq we then have
\eq \label{ConfScalarFlow}\breve{u}=\log(2+u)-\log(2-u) \eeq
and \eq \abs{D\breve{u}}^{2}=u^{-2}\sigma^{ij}u_{i}u_{j}\equiv\abs{Du}^{2}. \eeq
Note that \eq \begin{aligned} d\-{s}^{2}&=\frac{1}{(1-\frac 14 r^{2})^{2}}(dr^{2}+r^{2}\sigma_{ij}dx^{i}dx^{j})\\
							&=e^{2\psi}(dr^{2}+r^{2}\sigma_{ij}dx^{i}dx^{j}) .\end{aligned}\eeq
Let \eq\~{\t}=\frac 12\frac{r}{1-\frac 14 r^{2}},\eeq
then the second fundamental forms $\breve{h}^{i}_{j}$ and $h^{i}_{j}$ satisfy the relation
\eq e^{\psi}\breve{h}^{i}_{j}=h^{i}_{j}+v^{-1}\~{\t}\d^{i}_{j}\equiv \check{h}^{i}_{j}, \eeq 
cf. \cite[(5.11), (5.13)]{cg:ImcfHyp}. Set \eq g_{ij}=u_{i}u_{j}+u^{2}\sigma_{ij} \eeq and
\eq \check{h}_{ij}=g_{ik}\check{h}^{k}_{j}, \eeq
then the flow in $\H^{n+1},$
\eq \dot{x}=F^{-p}\breve{\nu}, \ F=F(\breve{h}^{i}_{j}), \eeq
now reads in $\R^{n+1}$ 
\eq \dot{x}=F^{-p}e^{(p-1)\psi}\nu, \label{confFlow} \eeq
where \eq F=F(\check{h}_{ij})=F(\check{h}^{i}_{j}). \eeq
Using \eq h_{ij}v^{-1}=-u_{;ij}+\-{h}_{ij} \eeq
 and the homogeneity of $F=F(\check{h}_{ij}),$ we obtain
 \eq \label{confscalarflow} \dot{u}-\dot{\Phi}F^{ij}u_{;ij}=-e^{(p-1)\psi}v\Phi+v^{-1}\dot{\Phi}F-\dot{\Phi}v^{-2}\~{\t}F^{ij}g_{ij}-\dot{\Phi}F^{ij}\-{h}_{ij}. \eeq
Here and in the following, $u_{;ij}$ denotes covariant differentiation with respect to $g_{ij},$ where merely indices, $u_{ij},$ denote derivatives with respect to $\sigma_{ij}$ and $\dot{u}=\frac{\del u}{\del t}$ is a partial derivative. We want to use coordinates $(x^{i}).$ 

\begin{lemss} \label{confDer}
Let $u$ be the scalar solution of (\ref{confFlow}). Then
\eq D^{m}u=\O^{-1+\e}\ \  \forall m\in\mathbb{N}^{*}\ \ \forall \e>0. \eeq
\end{lemss}

\begin{proof}
We have \eq\breve{u}_{i}=\frac{u_{i}}{1-\frac 14 u^{2}}.\eeq
In view of (\ref{ConfScalarFlow}) there holds\eq 2-u=(2+u)e^{-\breve{u}} \eeq and thus 
\eq \label{confDer1} (2-u)^{\b}=\O^{-\b}\ \ \forall \b\in\R,\eeq using Lemma \ref{oscBound}.
Set \eq g(u)=\frac{1}{1-\frac 14 u^{2}}\equiv \frac{1}{f(u)}.\eeq
Then
\eq D^{m}g=\sum_{k_{1}+\ldots+mk_{m}=m}c_{m}f^{-(\sum_{i=1}^{m}k_{i}+1)}\prod_{i=1}^{m}\left(\frac{D^{i}f}{i!}\right)^{k_{i}},\eeq
and \eq D^{i}f=\sum_{s+r=i}D^{s}u\star D^{r}u.\eeq
Taking $\abs{D\breve{u}}$ with respect to the spherical norm, we see that the claim holds for $m=1,$ by Lemma \ref{uandphi}. Suppose the claim to be true for $1\leq k\leq m-1.$ Then
\eq D^{m}\breve{u}=gD^{m}u+Du\star D^{m-1}g+\sum_{\substack{s+r=m-1\\ s,r\geq 1}}D^{r+1}u\star D^{s}g\eeq
so that
\eq \begin{aligned} D^{m}u&= g^{-1}D^{m}\breve{u}+g^{-1}Du\star D^{m-1}g+g^{-1}\sum_{\substack{s+r=m-1\\s,r\geq 1}}D^{r+1}u\star D^{s}g\\
				&=\O^{m(1-\g)-1}+\O^{\e m-1} \ \forall \gamma<1\ \forall\e>0. \end{aligned}\eeq 
\end{proof}

\begin{lemss} \label{addConfDer}
For all $m\in\mathbb{N}^{*}$ and for all $\e>0$ there hold
\eq \label{addConfDer1} D^{m}v^{\b}=\O^{-2+\e}\ \ \forall \b\in\R, \eeq
\eq  \label{addConfDer2} D^{m}\left(\frac{2u}{2+u}\right)=\O^{-1+\e}=D^{m}\left(\left(\frac{4}{2+u}\right)^{p-1}\right), \eeq
\eq \label{addConfDer3} D^{m}((2-u)^{\b})=\O^{-\b+\e m},\eeq
\eq \label{addConfDer4} D^{m}(\check{h}^{i}_{j}(2-u))=\O^{-1+\e}\eeq and \eq \label{addConfDer5} D^{m}(\check{h}_{ij}(2-u))=\O^{-1+\e}.\eeq
\end{lemss}

\begin{proof}
We consider
\eq v=\sqrt{1+u^{-2}\sigma^{ij}u_{i}u_{j}}. \eeq
Differentiation gives
\eq \begin{aligned} v_{i_{1}}&=\frac{1}{2v}(2u^{-2}\sigma^{kl}u_{ki_{1}}u_{l}-2u^{-3}\sigma^{kl}u_{k}u_{l}u_{i_{1}})\\
					&=v^{-1}(u^{-2}\sigma^{kl}u_{ki_{1}}u_{l}-u^{-3}\sigma^{kl}u_{k}u_{l}u_{i_{1}})\\
					&=\O^{-2+\e}\ \forall \e>0. \end{aligned}\eeq
Thus $D(v^{\b})=\b v^{\b-1}Dv=\O^{-2+\e}\ \forall\e>0.$ Let the claim hold for $1\leq k\leq m-1.$ Then
\eq \begin{aligned} D^{m}v&=\sum_{s+r=m-1}D^{s}(v^{-1})\star D^{r}(u^{-2}\sigma^{kl}u_{ki_{1}}u_{l}-u^{-3}\sigma^{kl}u_{k}u_{l}u_{i_{1}})\\
				&=\O^{-2+\e},\end{aligned}\eeq
so that
\eq D^{m}(v^{\b})=\sum_{k_{1}+\ldots+mk_{m}=m}c_{m}\star \O^{0}\star\prod_{i=1}^{m}\left(\frac{D^{i}v}{i!}\right)^{k_{i}}=\O^{-2+\e}. \eeq
Thus (\ref{addConfDer1}) is true.\ \\

To prove (\ref{addConfDer2}), suppose that $f$ is smooth, then
\eq D^{m}(f\circ u)=\sum_{k_{1}+\ldots +mk_{m}=m}c_{m}f^{(k)}(u)\prod_{i=1}^{m}\left(\frac{D^{i}u}{i!}\right)^{k_{i}}=\O^{-1+\e}, \eeq
$k=\sum_{i=1}^{m}k_{i},$ since in case \eq f(x)=\frac{2x}{2+x}\eeq or \eq f(x)=\left(\frac{4}{2+x}\right)^{p-1}\eeq we have
\eq D^{k}f\in C^{\infty}(u([0,\infty)\times\S^{n})).\eeq
In case of (\ref{addConfDer3}) we have \eq f(x)=(2-x)^{\b}\eeq such that
\eq f^{(k)}(x)=\prod_{i=0}^{k-1}(\b-i)(2-x)^{\b-k}(-1)^{k},\eeq
 implying \eq f^{(k)}(u)=\O^{k-\b}.\eeq
 Thus \eq D^{m}(f\circ u)=\O^{-\b+\e m}.\eeq
 In order to show (\ref{addConfDer4}), first observe that there holds, according to (\ref{longtime1}),
 \eq h^{i}_{j}=\frac{1}{vu}\d^{i}_{j}+\frac{1}{v^{3}u^{3}}u^{i}u_{j}-\frac{\sigma^{ik}-v^{-2}u^{-2}u^{i}u^{k}}{vu^{2}}u_{kj}. \eeq
 Have in mind, that now $\t{(u)}=u,$ $u^{i}=\sigma^{ik}u_{k}$ and derivatives are taken with respect to $\sigma_{ij}.$
 Thus
 \eq \begin{aligned}
 D^{m}(\check{h}^{i}_{j}(2-u))&=D^{m}\Big(\frac{2-u}{vu}\d^{i}_{j}+\frac{2-u}{v^{3}u^{3}}u^{i}u_{j}\\
 					&\hphantom{=}-(2-u)\frac{\sigma^{ik}-v^{-2}u^{-2}u^{i}u^{k}}{vu^{2}}u_{kj}+v^{-1}\frac{2u}{2+u}\d^{i}_{j}\Big)\\
					&=\O^{-1+\e}+\O^{-3+\e}+\O^{-2+\e}\\
					&=\O^{-1+\e}. \end{aligned} \eeq
(\ref{addConfDer5}) follows from \eq D^{m}(g_{ij})=D^{m}(u_{i}u_{j}+u^{2}\sigma_{ij})=\O^{-1+\e}.\eeq
\end{proof}

\vspace{0,01 cm}

\begin{thmss}
Let $u$ be the scalar solution of (\ref{confFlow}), then
\eq  D^{m}u=\O^{-1}\ \ \forall m\in\mathbb{N}^{*}. \eeq
\end{thmss}

\begin{proof}
We follow the corresponding proof in \cite[Thm. 6.11]{cg:ImcfHyp}.\\

Define \eq \phi=(2-u)^{-1}, \ \~{\phi}=\phi e^{-\frac{t}{n^{p}}}\eeq and
\eq \~{F}=F(\check{h}^{k}_{l}(2-u)),\ \~{\Phi}=\Phi(\~{F}).\eeq
There holds, having in mind that $h^{i}_{j}\rightarrow \frac 12 \d^{i}_{j},$ and using (\ref{confDer1}) as well as Theorem \ref{optGradDecay}, that
\eq \abs{\check{h}^{k}_{l}(2-u)-\d^{k}_{l}}\leq \abs{h^{k}_{l}(2-u)}+\left|\left(v^{-1}\frac{2u}{2+u}-1\right)\d^{k}_{l}\right|\leq ce^{-\frac{t}{n^{p}}}.\eeq
We have \eq \frac{\del \~{\phi}}{\del t}=\dot{\~{\phi}}=\left(\frac{\dot{u}}{2-u}-\frac{1}{n^{p}}\right)\~{\phi}, \eeq
\eq \~{\phi}_{ij}=\frac{u_{ij}}{2-u}\~{\phi}+\frac{2u_{i}u_{j}}{(2-u)^{2}}\~{\phi}\eeq
and thus
\eq \begin{aligned}
&\dot{\~{\phi}}-v^{-2}\dot{\~{\Phi}}\~{\phi}^{-(p+1)}e^{-\frac{p+1}{n^{p}}t}\~{F}^{ij}\~{\phi}_{ij}\\
=&\frac{\~{\phi}}{2-u}\Big(\dot{u}-v^{-2}\dot{\~{\Phi}}\~{\phi}^{-(p+1)}e^{-\frac{p+1}{n^{p}}t}\~{F}^{ij}u_{ij}\\
&\hphantom{=}-v^{-2}\frac{2}{2-u}\dot{\~{\Phi}}\~{\phi}^{-(p+1)}e^{-\frac{p+1}{n^{p}}t}\~{F}^{ij}u_{i}u_{j}-\frac{2-u}{n^{p}}\Big). \end{aligned}\eeq

An easy calculation shows
\eq \begin {aligned}
u_{ij}&= v^{2}u_{;ij}-u^{-1}(\sigma^{kl}u_{k}u_{l}\sigma_{ij}-2u_{i}u_{j})\\
	&=-vh_{ij}+v^{2}\-{h}_{ij}-u^{-1}(\sigma^{kl}u_{k}u_{l}\sigma_{ij}-2u_{i}u_{j})\\
	&=-v\check{h}_{ij}+\~{\t}g_{ij}+v^{2}\-{h}_{ij}-u^{-1}(\sigma^{kl}u_{k}u_{l}\sigma_{ij}-2u_{i}u_{j}). \end{aligned}\eeq
	
Thus we conclude
\eq \begin{aligned}
&\dot{\~{\phi}}-v^{-2}\dot{\~{\Phi}}\~{\phi}^{-(p+1)}e^{-\frac{p+1}{n^{p}}t}\~{F}^{ij}\~{\phi}_{ij}\\
=&\frac{\~{\phi}}{2-u}\Big(\dot{u}-\dot{\Phi}F^{ij}u_{;ij}+u^{-1}v^{-2}\dot{\~{\Phi}}\~{\phi}^{-(p+1)}e^{-\frac{p+1}{n^{p}}t}\~{F}^{ij}(\sigma^{kl}u_{k}u_{l}\sigma_{ij}\\
&\hphantom{=}-2u_{i}u_{j})-v^{-2}\frac{2}{2-u}\dot{\~{\Phi}}\~{\phi}^{-(p+1)}e^{-\frac{p+1}{n^{p}}t}\~{F}^{ij}u_{i}u_{j}-\frac{2-u}{n^{p}}\Big),\end{aligned} \eeq which is
			\eq \begin{aligned}	&v^{-1}\dot{\~{\Phi}}\~{F}(2-u)^{p-1}\~{\phi}-v\~{\Phi}(e^{\psi}(2-u))^{p-1}\~{\phi}\\
				&\hphantom{=}-v^{-2}\dot{\~{\Phi}}(2-u)^{p-1}\frac{2u}{2+u}F^{ij}g_{ij}\~{\phi}-\dot{\~{\Phi}}F^{ij}\-{h}_{ij}(2-u)^{p-1}e^{-\frac{t}{n^{p}}}\\
				&\hphantom{=}-\left(2u^{-1}+\frac{2}{2-u}\right)(v^{-2}\dot{\~{\Phi}}(2-u)^{p}\~{F}^{ij}u_{i}u_{j})\~{\phi}\\
				&\hphantom{=}+u^{-1}v^{-2}\dot{\~{\Phi}}(2-u)^{p}\~{F}^{ij}\sigma^{kl}u_{k}u_{l}\sigma_{ij}\~{\phi}-\frac{1}{n^{p}}\~{\phi},\end{aligned} \eeq
being equal to \eq \begin{aligned}				&\Big(v^{-1}\dot{\~{\Phi}}\~{F}-v^{-2}\dot{\~{\Phi}}\frac{2u}{2+u}F^{ij}g_{ij}\Big)(2-u)^{p-1}\~{\phi}\\
				&\hphantom{=}-\left(v\~{\Phi}\left(\frac{4}{2+u}\right)^{p-1}+\frac{1}{n^{p}}\right)\~{\phi}-\dot{\~{\Phi}}F^{ij}\-{h}_{ij}(2-u)^{p-1}e^{-\frac{t}{n^{p}}}\\
				&\hphantom{=}-\left(2u^{-1}+\frac{2}{2-u}\right)(v^{-2}\dot{\~{\Phi}}(2-u)^{p}\~{F}^{ij}u_{i}u_{j})\~{\phi}\\
				&\hphantom{=}+u^{-1}v^{-2}\dot{\~{\Phi}}(2-u)^{p}\~{F}^{ij}\sigma_{ij}\sigma^{kl}u_{k}u_{l}\~{\phi}. \end{aligned}\eeq
Set \eq w=\frac 12\abs{D^{m}\~{\phi}}^{2}.\eeq
Then by Lemma \ref{addConfDer} we have
\eq D^{m}\~{\phi}=\O^{\e m}\ \ \forall m\in\mathbb{N}^{*}\ \ \forall \e>0.\eeq
Differentiating the equation for $\~{\phi}$ covariantly with respect to $\sigma_{ij}$ $m$ times, we obtain
\eq \begin{aligned}
\dot{w}-v^{-2}\dot{\~{\Phi}}\~{\phi}^{-(p+1)}e^{-\frac{p+1}{n^{p}}t}\~{F}^{ij}w_{ij}&= \O^{-p}\star w+\O^{\e-p+3\e m}\\
												&\hphantom{=}+\O^{-1+\e+\e m}+\O^{-p+\e m}\\
												&\hphantom{=}+\O^{-(p+1)}\star D^{m+1}\~{\phi}\star D^{m+1}\~{\phi}\\
												&=\O^{-\d},\ \d>0, \end{aligned}\eeq
where first $\e$ has to be chosen in dependence of $p$ and $m.$
Thus \eq \~{w}=\sup\limits_{x\in\S^{n}}w(\cdot,x)\eeq
satisfies \eq\dot{\~{w}}\leq c_{m,\d}e^{-\d t} \eeq
and is bounded.\ \\

Thus \eq D^{m}\phi=\O^{1}\ \  \forall m\in\mathbb{N}. \eeq
This yields \eq Du=(2-u)^{2}D\phi=\O^{-1}.\eeq

If \eq D^{k}u=\O^{-1}\ \  \forall 1\leq k\leq m-1,\eeq
then
$$ \begin{aligned}
D^{m}\phi&=\sum_{k_{1}+\ldots +mk_{m}=m}\frac{m!}{k_{1}!\cdots k_{m}!}\frac{1}{(2-u)^{1+k}}\prod_{i=1}^{m-1}\left(\frac{D^{i}u}{i!}\right)^{k_{i}}\left(\frac{D^{m}u}{m!}\right)^{k_{m}}\\
		&=\frac{D^{m}u}{(2-u)^{2}}+\O^{1}, \end{aligned}$$
		
which implies \eq D^{m}u=\O^{-2}\star D^{m}\phi+\O^{-1}=\O^{-1}.\eeq

\end{proof}

\begin{corss}
The rescaled functions
\eq \~{u}=(u-2)e^{\frac{t}{n^{p}}} \ \mathrm{in}\ \R^{n+1}\eeq and
\eq\~{\breve{u}}=\breve{u}-\frac{t}{n^{p}}\ \mathrm{in}\ \H^{n+1} \eeq
are uniformly bounded in $C^{m}(\S^{n})$ for all $m\in\mathbb{N}$ and converge in $C^{\infty}(\S^{n})$ to a uniquely determined limit $\~{u}$ or $\~{\breve{u}}$ respectively.
\end{corss}

\begin{proof}
We follow the proof of \cite[Thm. 6.11]{cg:ImcfHyp}.\\

Because of the boundedness we only have to show, that the pointwise limit
\eq \lim\limits_{t\rightarrow\infty}(u(t,x)-2)e^{\frac{t}{n^{p}}} \eeq
exists for all $x\in\S^{n}.$
We have
\eq \begin{aligned} \dot{\~{u}}=\frac{\del\~{u}}{\del t}&=e^{(p-1)\psi}\frac{v}{F^{p}}e^{\frac{t}{n^{p}}}+\frac{1}{n^{p}}\~{u}\\
									&=\frac{u+2}{4}(2-u)e^{\frac{t}{n^{p}}}v\frac{4^{p}}{(2+u)^{p}}\~{F}^{-p}+\frac{1}{n^{p}}\~{u}\\
									&=\left(\frac{1}{n^{p}}-\frac{4^{p-1}}{(2+u)^{p-1}}v\~{F}^{-p}\right)\~{u}\\
									&\geq -ce^{-\frac{t}{n^{p}}}.\end{aligned}\eeq
Thus \eq (\~{u}-n^{p}ce^{-\frac{t}{n^{p}}})'\geq 0,\eeq
which implies the result.
\end{proof}

\begin{corss}
The limit function
\eq \~{\breve{u}}=\lim\limits_{t\rightarrow\infty}\frac{\breve{u}}{t}\eeq is constant in $\H^{n+1}$.
\end{corss}

\providecommand{\bysame}{\leavevmode\hbox to3em{\hrulefill}\thinspace}
\providecommand{\href}[2]{#2}

\end{document}